\documentclass[a4paper, 12pt]{amsart}
\usepackage{amsmath,mathrsfs,amsthm}
\usepackage{txfonts,rotating,tikz}
\usetikzlibrary{intersections,shapes.arrows,calc}

\usepackage[all,cmtip]{xy}
\usepackage{amssymb}
\usepackage{amsfonts}
\usepackage{amscd}
\usepackage{color} 
\usepackage[utf8]{inputenc}
\usepackage{hyperref}
\usepackage{fullpage}
\usepackage{epic}
\usepackage{mathtools}
\usepackage{tikz}
\usepackage{tikz-cd}
\usetikzlibrary{chains}
\usepackage[thinlines]{easytable}
\usepackage{array}
\usepackage{ctable}
\usepackage{pbox}
\usepackage[usestackEOL]{stackengine}
\usepackage[round,sort,comma,numbers]{natbib}
\newcolumntype{?}{!{\vrule width 1pt}}

\theoremstyle{Theorem}

\theoremstyle{Theorem}

\theoremstyle{Theorem}

\newtheorem*{theorem*}{Theorem}

\theoremstyle{Definition}

\newtheorem{theorem}{Theorem}[section]

\newtheorem{lemma}[theorem]{Lemma}
\newtheorem{proposition}[theorem]{Proposition}
\newtheorem{corollary}[theorem]{Corollary}

\newtheorem{remark}[theorem]{Remark}

\theoremstyle{definition}
\newtheorem{definition}[theorem]{Definition}

\newtheorem{example}[theorem]{Example}

\newcommand{\be}{\begin{equation}}
\newcommand{\ee}{\end{equation}}
\newcommand{\bt}{\begin{theorem}}
\newcommand{\et}{\end{theorem}}
\newcommand{\bd}{\begin{definition}}
\newcommand{\ed}{\end{definition}}
\newcommand{\bp}{\begin{proposition}}
\newcommand{\ep}{\end{proposition}}
\newcommand{\bl}{\begin{lemma}}
\newcommand{\el}{\end{lemma}}
\newcommand{\bco}{\begin{corollary}}
\newcommand{\eco}{\end{corollary}}
\newcommand{\br}{\begin{remark}}
\newcommand{\er}{\end{remark}}
\newcommand{\bex}{\begin{example}}
\newcommand{\eex}{\end{example}}
\newcommand{\ben}{\begin{enumerate}}
\newcommand{\een}{\end{enumerate}}
\newcommand{\bc}{\begin{cases}}
\newcommand{\ec}{\end{cases}}

\newcommand{\bpf}{\begin{proof}}
\newcommand{\epf}{\end{proof}}

\newcommand{\bma}{\begin{bmatrix}}
\newcommand{\ema}{\end{bmatrix}}

\newcommand{\fh}{\mathfrak{h}}

\newcommand{\fb}{\mathfrak{b}}
\newcommand{\fg}{\mathfrak{g}}

\newcommand{\li}{\,|\,}
\newcommand{\scr}{\mathscr}
\newcommand{\bb}{\mathbb}
\newcommand{\cal}{\mathcal}

\DeclareMathOperator{\Gr}{\mathtt{Gr}}

\DeclareMathOperator{\X}{\mathtt{X}}
\DeclareMathOperator{\Y}{\mathtt{Y}}
\DeclareMathOperator{\I}{\mathtt{I}}

\newcommand{\Lg}{{^L}\tilde{\fg} }
\newcommand{\Lh}{{^L}\tilde{\fh} }
\newcommand{\Lb}{{^L}\tilde{\fb} }

\newcommand{\beqn}{\begin{equation}}
\newcommand{\eeqn}{\end{equation}}

\begin{document}
\title{ A combinatorial study of affine Schubert varieties in the affine Grassmannian}
\author{Marc Besson}
\address{Department of Mathematics, University of North Carolina at Chapel Hill, Chapel Hill, NC 27599-3250, U.S.A.}
\email{marmarc@live.unc.edu}
\author{Jiuzu Hong}
\address{Department of Mathematics, University of North Carolina at Chapel Hill, Chapel Hill, NC 27599-3250, U.S.A.}
\email{jiuzu@email.unc.edu}
\maketitle
\begin{abstract}
 Let $\overline{\X}_\lambda$ be the closure of the $\I$-orbit $\X_\lambda$ in the affine Grassmanian $\Gr$ of a simple algebraic group $G$ of adjoint type, where $\I$ is the Iwahori subgroup and $\lambda$ is a coweight of $G$. 
 We find a simple algorithm which describes the set $\Psi(\lambda)$ of all $\I$-orbits  in $\overline{\X}_\lambda$ in terms of coweights. 
 We introduce $R$-operators (associated to positive roots) on the coweight lattice of $G$, which exactly describe the closure relation of $\I$-orbits. These operators satisfy Braid relations generically on the coweight lattice. We also establish a duality between the set $\Psi(\lambda)$ and the weight system of the level one affine Demazure module $\hat{\scr{D}}_\lambda$ of $^L\tilde{\fg}$ indexed by $\lambda$, where $^L\tilde{\fg}$ is the affine Kac-Moody algebra dual to the affine Kac-Moody Lie algebra $\tilde{\fg}$  associated to  the Lie algebra $\fg$ of $G$. 

\end{abstract}

\section{Introduction}
It is well-known that Schubert cells in the flag variety of a reductive group $G$ can be parametrized by the elements of  the Weyl group of $G$. Moreover, the closure relations among Schubert cells can be described by the Bruhat order on the Weyl group. There is another equivalent description of the Bruhat order in terms of the containment relations of Demazure modules of a Borel subgroup of $G$, which is established in the celebrated work \citep{BGG} by Bernstein-Gelfand-Gelfand. There is a generalization of this perspective for general Kac-Moody groups, see \citep{Ku}.

From now on throughout this paper, we assume that $G$ is a simple algebraic group over $\mathbb{C}$ of adjoint type. Let $\Gr$ denote  the affine Grassmannian $G(\scr{K})/G(\mathscr{O})$ of $G$, where $\scr{K}$ is the field of Laurent series over $\mathbb{C}$ and $\scr{O}$ is the ring of formal power series. The $G(\scr{O})$-orbits are indexed by dominant coweights of $G$. It is well-known that (cf.\,\cite[\S 4.5]{BD}, \cite[\S 2.1]{Zhu})
\begin{equation}\label{sec1_formula}  \Gr_\mu\subset \overline{\Gr}_\lambda  \text{ if and only if } \lambda-\mu \text{ is a sum of positive coroots of } G, \end{equation} 
where $\Gr_\mu$ and $\Gr_\lambda$ denote $G(\scr{O})$-orbits indexed by dominant coweights $\lambda, \mu$, and $\overline{\Gr}_\lambda$ is the closure of $\Gr_\lambda$. Moreover, the intersection cohomology of $\overline{\Gr}_\lambda$ carries an action of the Langlands dual group $^L{G}$ of $G$, which is irreducible and of highest weight $\lambda$ (cf.\,\cite{MV}).

In this paper, we consider the action of the Iwahori subgroup $\I$ on $\Gr$. The $\I$-orbits in $\Gr$ can be indexed by coweights of $G$. For each coweight $\lambda$, we denote by $\X_\lambda$ the associated $\I$-orbit and $\overline{\X}_\lambda$ the closure of $\X_\lambda$ in $\Gr$. For any two coweights $\lambda, \mu$, we introduce the partial order $\mu\prec_I \lambda$ if $\X_\mu\subset \overline{\X}_\lambda$.  This partial order naturally appears in the study of parabolic Kazhdan-Lusztig polynomials for affine Kac-Moody groups (cf.\,\cite{KT}), as well as in the context of non-symmetric Macdonald polynomials (cf.\,\cite{Bo1,Bo3,Bo4}.  For a given coweight $\lambda$, we denote by $\Psi(\lambda)$ the set of all coweights $\mu$ such that $\mu\prec_I \lambda$. In this paper, we describe the partial order $\prec_I$ in a way similar to the condition in (\ref{sec1_formula}), and describe the set $\Psi(\lambda)$ by a simple algorithm. We also introduce $R$-operators which help to describe the extremal elements in $\Psi(\lambda)$, and they themselves satisfy braid relations generically on the the weight lattice. Moreover, we give a representation theoretic interpretation of $\Psi(\lambda)$ in terms of level one affine Demazure modules, which are of twisted type when $G$ is not simply-laced.

In Lemma \ref{sect_lem_1} we describe an algorithm that is used to produce new elements for the set $\Psi(\lambda)$, and in Theorem \ref{Thm_1} we prove that this algorithm indeed produces all elements of $\Psi(\lambda)$. Roughly speaking, any element in $\Psi(\lambda)$ can be obtained by successively adding or subtracting positive coroots, depending on the signs of the pairing between coweights and positive roots. 

The set $\pi_0(\Gr)$ of components of $\Gr$ can be identified with the quotient group $\check{P}/\check{Q}$, where $\check{Q}$ is the coroot lattice of $G$. Let $\Gr^\kappa$ be a component of $\Gr$ that contains the $T$-fixed point $L_{-\check{\omega }_\kappa }$ associated to the coweight $-\check{\omega }_\kappa $, where $\check{\omega }_\kappa$ is a miniscule coweight or zero. In Proposition \ref{sect2.4_prop}, we explicitly realize the component $\Gr^\kappa$ as a partial flag variety $\tilde{G}_{\sf{sc}}/\tilde{P}_\kappa$ of the affine Kac-Moody group $\tilde{G}_{\sf{sc}}$ associated to the simply-connected cover of $G$, and realize each $\I$-orbit $\X_\lambda$ as an affine Schubert cell in $\tilde{G}_{\sf{sc}}/\tilde{P}_\kappa$.
In this way, we translate the partial order $\prec_I$ on coweights into the partial Bruhat order on $W_{\sf{aff}}/W_\kappa$, where $W_\kappa$ is the Weyl group of $\tilde{P}_\kappa$. In Section \ref{sect_Weyl_group}, we explictly realize the affine Weyl group $W_{\sf{aff}}$ as the Weyl group of the affine Kac-Moody algebra $\tilde{\fg}$ associated to $\fg$, and translate the Bruhat order on $W_{\sf{aff}}$ into certain conditions on coweights (cf.\,Proposition \ref{sect2.4_prop1}, Corollary \ref{sect2.3_cor1}). Combining all these preparations, in Section \ref{sect2.5} we prove Theorem \ref{Thm_1} that is described in the previous paragraph. In Section \ref{sect2.6}, we use length zero elements of the extended affine Weyl group to establish bijections between the sets $\Psi(\lambda)$ in different components of the affine Grassmannian (cf.\,Corollary \ref{sect2.6_cor}). \vspace{0.2cm}

Let $\tilde{\fg}$ be the affine Kac-Moody Lie algebra associated to the Lie algebra $\fg$ of $G$. Let $^L \tilde{\fg}$ be the affine Kac-Moody algebra with Dynkin diagram dual to that of $\tilde{\fg}$.
As already mentioned above, the components of $\Gr$ correspond to the miniscule coweights of $\fg$ or zero. Furthermore, the level one basic representations of $\Lg$ also correspond to the miniscule coweights of $\fg$ or zero (cf.\,Lemma \ref{level_one_weight_lem}). In summary, we have the following correspondences:
\begin{equation}
\label{int_1}
\Gr^\kappa  \longleftrightarrow    \check{\omega}_\kappa  \longleftrightarrow  \scr{H}_\kappa \,,
\end{equation}
where $\scr{H}_\kappa$ denotes the associated level one basic representation of $\Lg$. This suggests a duality between the affine Grassmannian $\Gr$ of $G$ and the level one representations of $\Lg$. For every coweight $\lambda\in \check{P}$, if $\X_\lambda\subset \Gr^\kappa$, we may associate a level one affine Demazure module  $\hat{\scr{D} }_\lambda$ generated by a maximal weight vector $v_{\varpi(\lambda)}\in \mathscr{H}_\kappa$ associated to $\lambda$. By Proposition \ref{Bruhat_Dem_affine_prop} together with Proposition \ref{sect2.4_prop}, 
we can establish the following correspondences: 
\begin{equation}
\label{int_2}
\fbox{$\mu\prec_I \lambda $} \longleftrightarrow   \fbox{partial Bruhat order on $W_{\sf{aff}}/W_\kappa$}  \longleftrightarrow    \fbox{$  \hat{\scr{D} }_\mu \subset  \hat{\scr{D} }_\lambda$}
\end{equation}
where in the first correspondence we view $W_{\sf{aff}}$ as the Weyl group of $\tilde{\fg}$, and in the second correspondence we view $W_{\sf{aff}}$ as the Weyl group of $\Lg$ (cf.\,Section \ref{sect_duality1}).
Under these correspondences, in Theorem \ref{duality_thm} we show that there exists a natural projection from the weight system of $ \hat{\scr{D} }_\lambda$ to the set $\Psi(\lambda)$. It is interesting to point out that we crucially use the Frenkel-Kac construction of level one basic representations of $\Lg$ in the proof of Theorem \ref{duality_thm}. Furthermore, in Lemma \ref{sect4.2_lem2} we give a representation-theoretic interpretation for the algorithm in Lemma \ref{sect_lem_1}. More precisely we show that the algorithm of adding or subtracting positive coroots to coweights in $\Psi(\lambda)$ can be interpreted as the actions of the positive real root vectors of $\Lg$ on weight vectors in $\hat{\scr{D} }_\lambda$. 
\vspace{0.2cm}

In Section \ref{sect4}, we study the set $\Psi(\lambda)$ and the partial order $\prec_I$ further. We first show that for any two coweights $\lambda, \mu\in \check{P}$, when they are located in the same chamber, then  $\mu\prec_I \lambda$ if and only if $\lambda-\mu$ is a sum of positive coroots relative to that chamber (cf.\,Theorem \ref{sect4.1_Thm}). This is analogous to the statement in (\ref{sec1_formula}). In Section \ref{sect_4.2}, we introduce $R$-operators on the coweight lattice $\check{P}$ associated to positive roots. It turns out that these operators exactly characterize the partial order $\prec_I$ (cf.\,Theorem \ref{R_op_Thm}). In Proposition \ref{Cover_prop_1} and Proposition \ref{Cover_prop_2}, we describe explicitly the covering relations of $\prec_I$ for a coweight $\lambda$ when $\lambda$ is mildly regular.
In Section \ref{braid_sect}, we show that for any two positive roots $\alpha, \beta$, when they generate a rank two root system as simple roots, then the associated $R$-operators $R_\alpha,R_\beta$ satisfy a braid relation when the coweights are away from certain critical hyperplanes (Proposition \ref{braid_prop}). We introduce in Definition \ref{regularity_def} the notion of $\alpha$-regularity of a coweight in a fixed chamber, where $\alpha$ is a positie root. This notion allows us to cross the wall $H_\alpha$ defined by $\alpha$ so that $R_\alpha(\lambda)$ is in the reflected chamber and $R_\alpha$ preserves the partial order $\prec_I$ (cf.\,Proposition \ref{stabilizer_fundamental_wt}). This allows us to produce an algorithm to describe the vertices of the convex hull of $\Psi(\lambda)$, or in other words the moment polytope of the affine Schubert variety $\overline{\X}_\lambda$, see the discussions in Section \ref{sect4.4}.
\vspace{0.1 in}

\noindent {\bf Acknowledgments}:   We would like to thank I.\,Cherednik, P.\,Fiebig and S.\,Kumar for helpful suggestions and discussions. 
 We would like to thank S.\,Nie for bringing the reference \cite{St} to our attentions. We also would like to thank Changlong Zhong for careful reading and pointing out some typos. 
J.\,Hong is partially supported by the Simons Foundation Collaboration Grant 524406.

\section{A combinatorial description of the closure relation of Iwahori orbits in affine Grassmannian}
\subsection{Notations} Let $G$ be a simple algebraic  group $G$ over $\mathbb{C}$ of adjoint type. Pick a maximal torus $T$ contained in a Borel subgroup $ B$ of  $ G$. Let $X^*(T)$ denote the lattice of characters of $T$ and $X_*(T)$ the lattice of cocharacters of $T$. Let $\Phi$ denote the set of roots for $T$, $\check{\Phi}$ the coroots so $(X_*(T), \Phi, X_*(T), \check{\Phi})$ is the root datum of $G$. Let $W$ denote the Weyl group of $G$. 
We denote by $\Phi^+$ (resp. $\check{\Phi}^+$) the set of positive roots (resp. positive coroots) determined by $T \subset B$. Let $\alpha_1, \cdots, \alpha_\ell$ (resp. $\check{\alpha}_1,\cdots, \check{\alpha}_\ell$) be simple roots (resp. simple coroots).  Under the assumption that $G$ is of adjoint type, the lattice $X_*(T)$ coincides with the coweight lattice $\check{P}$.

 Let $\mathscr{O}=\mathbb{C}\llbracket t \rrbracket$ be the formal power series in $t$ with coefficients in $\mathbb{C}$, and let $\mathscr{K}=\mathbb{C}((t))$ be the field of formal Laurent series in $t$. Let $\Gr$ denote the affine Grassmanian $G(\mathscr{K})/G(\mathscr{O})$ of $G$. We have an evaluation map \[{\rm ev}_0: G(\mathscr{O}) \rightarrow G\] sending $t \mapsto 0$. Write $\I={\rm ev}_0^{-1}(B)$ as the Iwahori subgroup of $G(\mathscr{K})$. Any cocharacter $\lambda: \mathbb{C}^*\to T$ gives rise to an element $t^\lambda\in G( \mathscr{K})$. Set $L_\lambda:= t^\lambda G(\mathscr{O})/G(\mathscr{O})\in \Gr$. Then all 
 $T$-fixed points in $\Gr$ are given by $L_\lambda$, where $\lambda\in X_*(T)$. We denote by $\X_\lambda$ the $\I$-orbit $\I\cdot L_\lambda$ in $\Gr$, and we denote by $\Gr_{\lambda}$ the $G(\mathscr{O})$-orbit  $G(\mathscr{O})\cdot L_\lambda$.
The variety $\X_\lambda$  has a unique $T$-fixed point $L_\lambda$, and $\Gr_\lambda$ has $T$-fixed points $L_{w(\lambda)}$, for $w\in W$. 



\begin{definition} Let $\Psi(\lambda)=\{L_\mu \in \check{P}\, | \, t^{\mu} \in \overline{\X}_\lambda \}$. If $\mu\in \Psi(\lambda)$, then we write $\mu\prec_I \lambda$. Clearly $\prec_I$ gives a partial order on $\check{P}$.
\end{definition}

\subsection{The algorithm }
\label{sect_algorithm}
We begin with a key lemma in order to describe the set $\Psi(\lambda)$.

\begin{lemma}
\label{sect_lem_1}
Let $\lambda \in \check{P}$, and let $\alpha \in \Phi^+$ be a positive root. 

1) If $\langle \lambda, \alpha \rangle >0$, then $\lambda-k \check{\alpha} \in \Psi(\lambda)$, for $1\leq k \leq \langle \lambda, \alpha \rangle$.

2) If $\langle \lambda, \alpha \rangle <0$, then $\lambda+k \check{\alpha}  \in \Psi(\lambda)$, for $1 \leq k \leq -\langle \lambda, \alpha \rangle -1$.
\end{lemma}

\begin{proof}
For any positive root $\alpha\in \Phi^+$, we may choose root subgroup homomorphisms $x_{\alpha}, x_{-\alpha}$ corresponding to the roots $\alpha, -\alpha$, and a cocharacter $h_\alpha:  \mathbb{C}^\times\to  G$, which give rise to a group homomorphism 
\[\phi:  {\rm SL}_2\to G  \]
such that 
\[  \phi\big ( \begin{bmatrix}  1 & a \\   0  &  1   \end{bmatrix}      \big) = x_\alpha(a), \quad  \phi \big ( \begin{bmatrix}  1 & 0 \\   a  &  1   \end{bmatrix}     \big ) = x_{-\alpha}(a), \quad \text{and } \quad \phi \big (\begin{bmatrix}   a & 0\\ 0 & a^{-1}     \end{bmatrix}  \big ) = h_\alpha(a). \]
Set 
\begin{equation}
\label{eq1.0}
n_\alpha(a):= x_\alpha(a)x_{-\alpha}(-a^{-1})x_\alpha(a).
\end{equation}
Then,
\[ \phi\big (  \begin{bmatrix}  0 & a  \\   -a^{-1}  &  0     \end{bmatrix}    \big) =n_\alpha(a). \]
  We also have the following equalities, 
\begin{equation}
\label{eq1.1}
  n_\alpha(a)=x_{-\alpha}(-a^{-1})x_\alpha(a) x_{-\alpha}(-a^{-1}), \quad   n_\alpha(ab)=h_\alpha(b)\cdot n_\alpha(a).
\end{equation}
 for any $a,b\in \mathbb{C}^*$.

Case 1:    $ \langle \lambda, \alpha \rangle >0$.

Let $k$ be any integer such that $1\leq k \leq \langle \lambda, \alpha \rangle $.
Consider the morphism $f_{\alpha,k, \lambda}: \mathbb{A}^1  \to  \Gr$ given by  $ a\mapsto   x_{\alpha}(at^{-k+  \langle \lambda, \alpha \rangle  })\cdot L_\lambda\in \X_\lambda$. Note that 
\begin{equation}
\label{eq1.2}  f_{\alpha,k,\lambda}(a)=t^\lambda x_\alpha(a t^{ -k   }) \cdot L_0. \end{equation}

As $ k> 0$, $f_{\alpha,k,\lambda}$ defines an $\mathbb{A}^1$-curve in $\X_\lambda$ passing through $L_\lambda$ when $a=0$. 
In view of  (\ref{eq1.1}), we can write 
\begin{equation}
\label{eq1.3}
 x_\alpha(at^{-k})=x_{-\alpha}(a^{-1}t^{k } ) t^{ -k \check{\alpha} }  n_\alpha(a) x_{-\alpha}(a^{-1}t^{k} ),\end{equation}
where $t^{ -k \check{\alpha} }=h_\alpha(t^{-k })$.
It follows that 
\[  
f_{\alpha,k,\lambda}(a) =t^\lambda  x_{-\alpha}(a^{-1}t^{k}) t^{ -k \check{\alpha} } \cdot L_0 
                                     =t^{\lambda- k \check{\alpha}  }x_{-\alpha}(  a^{-1} t^{-k} )\cdot L_0.
\]

When $a\to  \infty$, $f_{\alpha,k,\lambda}(a) \to  L_{  \lambda-k\check{\alpha} }$. It follows that $L_{  \lambda-k\check{\alpha} }\in \overline{\X}_\lambda$. In other words, $\lambda-k\check{\alpha}\in \Psi (\lambda) $.

Case 2:  $\langle \lambda, \alpha \rangle<0$.

Let $k$ be any integer such that $1\leq k \leq -\langle \lambda, \alpha \rangle -1$.
Consider the morphism $g_{\alpha,k, \lambda}: \mathbb{A}^1  \to  \Gr$ given by  $ a\mapsto   x_{-\alpha}(at^{-k- \langle \lambda, \alpha \rangle  })\cdot L_\lambda\in \X_\lambda$. Similar to (\ref{eq1.2}), we have
\begin{equation}
\label{eq1.4}  g_{\alpha,k,\lambda}(a)=t^\lambda x_{-\alpha}(a t^{ -k   }) \cdot L_0. \end{equation}
Thus, $g_{\alpha,k, \lambda}$ defines an $\mathbb{A}^1$-curve in $\X_\lambda$ passing through $L_\lambda$ when $a=0$. In view of (\ref{eq1.0}), we have
\begin{equation}
\label{eq1.5}
 x_{-\alpha}(at^{-k})=x_{\alpha}(a^{-1}t^{k } ) t^{ k \check{\alpha} }  n_\alpha(-a^{-1}) x_{-\alpha}(a^{-1}t^{k} ).\end{equation}
It follows that 
\[  g_{\alpha,k,\lambda}(a)=t^\lambda x_{\alpha}(a^{-1}t^{k } ) t^{ k \check{\alpha} }\cdot L_0= t^{\lambda+ k \check{\alpha}  } x_\alpha(a^{-1}t^{-k}  )\cdot L_0. \]

When $a\to \infty$, $g_{\alpha,k,\lambda}(a)\to L_{ \lambda+ k \check{\alpha}  }$. It follows that $ \lambda+ k \check{\alpha}  \in \Psi(\lambda)$.

\end{proof}
This lemma will provide an algorithm which completely describes the set $\Psi(\lambda)$. For a positive root $\alpha$, the algorithm will rely on the sign of  $\langle \lambda, \alpha \rangle$, In fact a representation-theoretic explanation will be given 
in Lemma \ref{sect4.2_lem2} of Section \ref{Sect_affine_Dem}.

For any coweight $\mu\in \check{P}$ and any positive root $\alpha\in \Phi^+$, we first introduce the following set of coweights attached to $\mu$ and $\alpha$
\begin{equation}
\label{sect2.1_set_S}
 S(\mu, \alpha)=   \begin{cases} \{  \mu-k\check{\alpha} \,|\, 0\leq k\leq  \langle \mu, \alpha  \rangle     \}, \mbox{ when }  \langle \mu, \alpha  \rangle\geq 0 \\
    \{\mu+k\check{\alpha}\,|\, 0\leq  k < -\langle \mu, \alpha \rangle \}, \mbox{ when   } \langle \mu, \alpha \rangle <0          \end{cases}.\end{equation}
By Lemma \ref{sect_lem_1}, $S(\mu, \alpha)$ is a subset of $\Psi(\lambda)$ if $\mu\in \Psi(\lambda)$.
We now define an increasing filtration  $\{ \Psi_i(\lambda) \}_{i\geq 0}$ of subsets in $\Psi(\lambda)$ as follows.
\begin{definition}
\label{sect2.2_def1}
Define $\Psi_0(\lambda)=\{\lambda\}$, and
 \[\Psi_i(\lambda)=   \bigcup_{\mu\in \Psi_{i-1}(\lambda), \alpha\in \Phi^+}  S(\mu, \alpha). \] 
 Let $\Psi_{\infty}(\lambda)$ denote the union of all $\Psi_i(\lambda)$. From this definition, we observe that if $\Psi_n(\lambda)=\Psi_{n+1}(\lambda)$ for some integer $n$, then $\Psi_\infty(\lambda)=\Psi_n(\lambda)$.
\end{definition}

\begin{lemma}
This filtration stabilizes after finite many steps, that is, there exists a positive integer $n$ such that $\Psi_{\infty}(\lambda)=\Psi_{n}(\lambda)$. 
\end{lemma}
\begin{proof}
Let $w$ be an element in the Weyl group $W$  such that $\lambda^+:=w(\lambda)$ is dominant.
Since $\X_\lambda \subset \Gr_{\lambda^+}$, we have 
\[  \Psi(\lambda)\subset \{ \mu \, | \, L_\mu \in (\Gr_{\lambda^+})^T \}. \]
Hence, $\Psi(\lambda)$ is a finite set, and so is $\Psi_\infty(\lambda)$. 
Therefore the filtration stabilizes after finite many steps.
\end{proof}

We are now ready to state the following theorem. 
\begin{theorem}
\label{Thm_1}
For any $\lambda\in \check{P}$, $\Psi(\lambda)=\Psi_\infty(\lambda)$.
\end{theorem}
This theorem gives an effective algorithm to describe the set $\Psi(\lambda)$. We will first make some preparations, and then the proof will be given in Section \ref{sect2.5}. 

\subsection{Affine Weyl group}
\label{sect_Weyl_group}
The Weyl group $W$ acts on the coroot lattice $\check{Q}$. From here we get an affine Weyl group $W_{\sf{aff}}:=\check{Q} \rtimes W$. We write elements of $W_{\sf{aff}}$ as $\tau_{\lambda}w$, where $\lambda\in \check{Q}, w\in W$. The element $\tau_{\lambda}w$ acts on $\check{P}$ by 
\[ \tau_\lambda w(\mu):=w(\mu)+\lambda, \quad \text{ for any } \mu\in \check{P}. \]
For any two elements $\tau_{\lambda_1}w_1,\tau_{\lambda_2}w_2$, the multiplication is given by
\begin{equation}
\label{multiplication_affine_Weyl}
  (\tau_{\lambda_1}w_1)\cdot (\tau_{\lambda_2} w_2)=\tau_{\lambda_1+w_1(\lambda_2)} w_1w_2. \end{equation}
The pair $(W_{\sf{aff}}, \hat{S})$ is a Coxeter system where  $\hat{S}$  consists of simple reflections $\{ s_i  \,|\, i\in I \}$   and a simple affine reflection $s_{0}=\tau_{\check{\theta}}s_\theta$ where $\theta$ is the highest positive root of $G$ and $\check{\theta}$ is the coroot of $\theta$. We denote by $\prec$ the Bruhat order on $(W_{\sf{aff}}, \hat{S})$.
 
 Let $\tilde{W}_{\sf{aff}}$  be the extended affine Weyl group $\check{P}\rtimes W$. The multiplication is given similarly as in (\ref{multiplication_affine_Weyl}). Following \cite{IM} we define the length function $\ell$  on $\tilde{W}_{\sf{aff}}$,
  \begin{equation}
  \label{length_formula}
   \ell(\tau_\lambda w)= \sum_{\alpha\in \Phi^+, w^{-1}(\alpha) \in \Phi^+   } |\langle\lambda, \alpha \rangle | + \sum_{\alpha\in \Phi^+, w^{-1}(\alpha) \in \Phi^-}|  \langle \lambda, \alpha \rangle-1  |. \end{equation}
This length function $\ell$ on $W_{\sf{aff}}$ coincides with the length function on the Coxeter system $(W_{\sf{aff}}, \hat{S})$.

Let $\fg$ be the Lie algebra of $G$. We  associate to $\fg$  the (completed) affine Kac-Moody algebra $\tilde{\fg}:=\fg\otimes \mathscr{K}\oplus \mathbb{C}c \oplus \mathbb{C}d$, where $c$ is the center, $d$ is the degree operator, and the Lie bracket is defined as in \cite[\S7.2]{Ka}.
 The affine Kac-Moody algebra $\tilde{\fg}$ corresponds to the extended Dynkin diagram $\hat{\Gamma}$ of $\fg$ with the set of vertices
 $\hat{I}=I\sqcup\{0\}$. The Cartan subalgebra $\tilde{\fh}$ of $\tilde{\fg}$ is given by $\fh\oplus \mathbb{C} c\oplus \mathbb{C} d$, where $\fh$ is the Cartan subalgebra of $\fg$. Let  $\tilde{\fh}^*$ denote  the linear dual  of $\tilde{\fh}$.
  Let $\delta$ denote the linear functional on $\tilde{\fh}$ such that 
\[ \delta|_{\fh}=0,\quad   \delta(c)=0,\quad  \delta ( d)=1.\]  
We first describe the affine root system associated to $\tilde{\fg}$. The set of all real affine roots of $\tilde{\fg}$ is given by 
\[ \hat{\Phi}_{\sf{re}}=\{  \alpha+k\delta  \,|\, \alpha\in \Phi, k\in \mathbb{Z}   \}, \]
where the set of positive affine roots is given by 
\[ \hat{\Phi}^+_{\sf{re}} =  \{  \alpha+k\delta  \,|\, \alpha\in \Phi, k> 0   \} \cup  \Phi^+, \]
and the set of negative affine roots is given by 
\[ \hat{\Phi}^+_{\sf{re}} =  \{  \alpha+k\delta  \,|\, \alpha\in \Phi, k< 0   \} \cup  \Phi^-. \]

Let $(\cdot |\cdot )$ denote the normalized bilinear form on $\tilde{\fh}$, and the induced bilinear form on $\tilde{\fh}^*$ (cf.\,\cite[\S6.1]{Ka}). Let $\nu: \tilde{\fh}\to \tilde{\fh}^*$ be the induced isomorphism. Then $\nu(c)=\delta$. Moreover, for any $\alpha\in \Phi$, $\nu(\check{\alpha})=\frac{2\alpha}{(\alpha| \alpha) }$ where $\check{\alpha}$ is the coroot of $\alpha$. We will denote by $\langle,\rangle$ the natural pairing between $\tilde{\fh}$ and $\tilde{\fh}^*$.

We may realize the affine Weyl group $W_{\sf{aff}}$ as the Weyl group of the affine Kac-Moody algebra $\tilde{\fg}$ in the sense of \cite[\S3.7]{Ka}, via  the action of $W_{\sf{aff}}$ on $\tilde{\fh}^*$. Following \cite[\S13.1]{Ku}, we define
\begin{equation} 
\label{formula_affine}
 \tau_\lambda w(x ) = w(x)+ \langle w(x), c  \rangle  \nu (\lambda) - (     \langle w(x), \lambda  \rangle+  \frac{1}{2} (\lambda |  \lambda)  \langle w(x), c  \rangle  )   \delta, \end{equation}
 for any $ \tau_\lambda w \in W_{\sf{aff}}, x\in \tilde{\fh}^*$.
\begin{lemma}
\label{section2.3_lem1}
The element $\tau_{k  \check{ \alpha} } s_\alpha\in W_{\sf{aff}}$ corresponds to the reflection on $\tilde{\fh}^*$ associated to the affine root $-\alpha+k\delta$. 
\end{lemma}
\begin{proof}
For any $x\in \tilde{\fh}^*$, we have
\begin{align*}
 \tau_{k\check \alpha}s_\alpha (x) &= x- \langle  x, \check{\alpha} \rangle \alpha+   \langle s_\alpha(x), c  \rangle  \nu (k  \check{\alpha }) - (     \langle s_\alpha(x), k\check{\alpha}  \rangle+  \frac{1}{2} (k\check{\alpha } | k \check{\alpha})  \langle s_\alpha(x), c  \rangle  )   \delta. \\
                                                      &=x- \langle  x, \check{\alpha} \rangle \alpha+  \frac{ 2k \langle x, c  \rangle}{  (\alpha|\alpha)  }  \alpha    - (    - k\langle x, \check{\alpha}  \rangle+  \frac{2k^2}{  (\alpha|\alpha) }  \langle x, c  \rangle  )   \delta.\\
                                                      &= x-( -  \langle x,\check{\alpha} \rangle+   \frac{ 2k \langle x, c  \rangle}{  (\alpha|\alpha)  }  )(-\alpha+k\delta)\\
                                                      &=x-\frac{ 2 (x|-\alpha+k\delta) }{(-\alpha+k\delta | -\alpha+ k\delta)} (-\alpha+k\delta),
\end{align*}
where the last equality holds since $\nu(c)=\delta$ and $(-\alpha+k\delta|-\alpha+k\delta)=(\alpha|\alpha)$. This is exactly the reflection  on $\tilde{\fh}^*$ associated to $-\alpha+k\delta$.
\end{proof}

For any $\alpha\in \Phi^+$ and $k\in \bb{Z}$, set 
\[ s_{\alpha,k}:= \tau_{-k\check{\alpha}} s_\alpha \in W_{\sf{aff}}.\]
  Then by Lemma \ref{section2.3_lem1}, $s_{\alpha,k}$ is the reflection associated to $\alpha+k\delta$. In particular $s_0=\tau_{\check{\theta} }s_\theta $ is the reflection associated to the affine simple root $\alpha_0:=-\theta+ \delta$. 

\begin{lemma}
\label{Coxeter_criterion}
Let $\alpha$ be a root in $\Phi$. Assume that $k\geq 0$ if $\alpha\in \Phi^+$, and $k> 0$ if $\alpha\in \Phi^-$.
 Then for any $\tau_{-\lambda}w\in W_{\sf{aff}}$, $s_{\alpha, k}\tau_{-\lambda} w\prec \tau_{-\lambda} w$ if and only if 
 \[   
 \begin{cases}
 k< \langle \lambda, \alpha  \rangle, \quad  \text{ when } w^{-1}(\alpha)\in \Phi^+ \\
 k\leq   \langle \lambda, \alpha  \rangle, \quad  \text{ when }   w^{-1}(\alpha)\in \Phi^-
 \end{cases}.
   \]
  \end{lemma}

\begin{proof}
We may realize $W_{\sf{aff}}$ as the Weyl group of the affine root system of $\tilde{\fg}$. Note that the assumption on $\alpha$ and $k$ is equivalent to that the affine root $\alpha+k\delta $ is positive. 
 In view of Lemma \ref{section2.3_lem1}, $s_{\alpha, k}$ corresponds to the reflection $s_{\alpha+k\delta}$. By a general fact of the theory of Coxeter groups (cf.\,\cite[\S 5.7]{Hu2}), $s_{\alpha, k}\tau_{-\lambda} w\prec \tau_{-\lambda} w$ if and only $w^{-1}\tau_\lambda(\alpha+k\delta )\in \hat{\Phi}^-_{\sf{re}}$. By the formula (\ref{formula_affine}), 
\[ w^{-1}\tau_\lambda(\alpha+k\delta )= w^{-1}(\alpha)+ (k-\langle \lambda, \alpha \rangle )\delta. \]
Thus the lemma immediately follows.
\end{proof}

\begin{proposition}
\label{sect2.3_prop1}
Let $\alpha$ be a positive root in $\Phi^+$. Assume that $s_{\alpha,k}\tau_{-\lambda}w\prec   \tau_{-\lambda }w$. 
\begin{enumerate}
\item   If $w^{-1}(\alpha)\in \Phi^+$ and $k\geq 0$, then $\langle \lambda, \alpha \rangle>0$, and $k< \langle \lambda, \alpha \rangle $.
\item If $w^{-1}(\alpha)\in \Phi^-$ and $k\geq 0$, then $\langle \lambda, \alpha \rangle\geq 0$, and $k\leq \langle \lambda, \alpha \rangle$. 
\item If $w^{-1}(\alpha) \in \Phi^+$ and $k<0$, then $\langle \lambda, \alpha \rangle<0$, and $k\geq \langle \lambda, \alpha \rangle$. 
\item  If $w^{-1}(\alpha) \in \Phi^-$ and $k<0$, then $\langle \lambda, \alpha \rangle<-1$, and $k> \langle \lambda, \alpha \rangle$. 
\end{enumerate}
\end{proposition}
\begin{remark}When $\alpha$ is a simple root or a highest root, a similar characterization appears in Lemma 1.3 in \citep{Bo1}.
\end{remark}
\begin{proof}
In first two cases, $s_{\alpha, k}$ corresponds to the reflection of the positive affine root $\alpha+k \delta$. In the last two cases, $s_{\alpha, k}$ corresponds to the reflection of the positive affine root $-\alpha-k\delta$. Then the proposition easily follows from Lemma \ref{Coxeter_criterion}.
\end{proof}

\begin{definition}
\label{Mini_coweight}
A coweight $\check{\omega }\in \check{P}$ is called miniscule, if  for any positive root  $\alpha\in \Phi^+$, $\langle \check{\omega }, \alpha\rangle \in \{0, 1\}$.
\end{definition}
 The following corollary and Proposition \ref{sect2.3_prop1} will be used in Section \ref{sect2.5}.
\begin{corollary}
\label{sect2.3_cor1}
Let $\check{\omega }$ be a miniscule coweight. 
For any $\lambda\in \check{Q}-\check{\omega }$, $\alpha\in \Phi^+$ and $y\in \tau_{\check{\omega} }W\tau_{-\check{\omega}} $, assume that $s_{\alpha,k} \tau_{-\lambda-\check{\omega }  } y \prec   \tau_{-\lambda-\check{\omega }  } y$. 
\begin{enumerate}
\item If $k\geq 0$, then $k\leq  \langle \lambda, \alpha \rangle $;
\item If $k< 0$, then  $k\geq   \langle \lambda, \alpha \rangle $.
\end{enumerate}
\end{corollary}
\begin{proof}
We can write $y=\tau_{ \check{\omega } } w \tau_{-\check{\omega }}$ for some $w\in W$. Note that 
\[  \tau_{-\lambda-\check{\omega }  } y=\tau_{-(\lambda+ w( \check{\omega } ))} w.\]
 We first prove part (1). Assume that $k\geq 0$. 
If $w^{-1}(\alpha)\in \Phi^+$, then $\langle  w( \check{\omega } ), \alpha \rangle =\langle   \check{\omega }, w^{-1}(\alpha) \rangle\in \{0, 1\} $. Moreover, 
  by part (1) of Proposition \ref{sect2.3_prop1}, 
\[ k<  \langle \lambda+ w( \check{\omega } ), \alpha \rangle =   \langle \lambda,\alpha \rangle+ \langle  w( \check{\omega } ), \alpha \rangle. \]
Hence $k\leq  \langle \lambda,\alpha \rangle$. If $w^{-1}(\alpha)\in \Phi^-$, then $\langle  w( \check{\omega } ), \alpha \rangle\in\{0, -1\}$. By  part (2) of Proposition \ref{sect2.3_prop1}, we get
\[ k\leq   \langle \lambda+ w( \check{\omega } ), \alpha \rangle =   \langle \lambda,\alpha \rangle+ \langle  w( \check{\omega } ), \alpha \rangle. \]
Hence we also have $k\leq  \langle \lambda,\alpha \rangle$. This shows that in case $k\geq 0$, we always have  $k\leq  \langle \lambda,\alpha \rangle$, no matter $w^{-1}(\alpha)$ is positive or negative.

By similar arguments, we can show that if $k< 0$, then $k\geq   \langle \lambda, \alpha \rangle $.
\end{proof}

\subsection{  Components of $\Gr$ as partial flag varieties of the Kac-Moody group}
Let $G_{\sf{sc}}$ be the simply-connected cover of $G$. Let $\tilde{G}_{\sf{sc}}$ be the affine Kac-Moody group with Lie algebra $\tilde{\fg}$ in the sense of Kumar and Mathieu (cf.\,\cite[\S VI]{Ku}), which can be realized as a central extension of the semi-direct product  $G_{\sf{sc}}(\mathscr{K})\rtimes \bb{C}^\times$ ($\bb{C}^\times$ acts on $G_{\sf{sc}}(\mathscr{K})$ by the loop rotation) (cf.\,\cite[Theorem 13.2.8]{Ku}). It is known that when $G$ is of adjoint type, the associated affine Grassmannian $\Gr$ has $|\check{P}/\check{Q}|$ components. In this subsection we produce an explicit description of each component of $\Gr$ as a partial flag variety of the Kac-Moody group $\tilde{G}_{\sf{sc}}$.

  Let $\hat{M}$ denote the set of vertices $i\in \hat{I} $ such that $a_i=1$, where $a_i$ is the Kac labeling of affine Dynkin diagram $\hat{\Gamma}$  \cite[p.54, Table Aff 1]{Ka}. Since $a_0=1$, $\hat{M}=M\cup \{0\}$, where $M:=\hat{M}\cap I$.
 Let $\theta$ be the highest root of $\fg$. It is known that 
  \begin{equation}
  \label{theta_formula}
  \theta=\sum_{i\in I} a_i \alpha_i. \end{equation}
  For each $i\in I$, let $\check{\omega}_i$ be the fundamental coweight of $\fg$ attached to $i$, i.e. for any simple root $\alpha_j$,
   \[\langle \check{\omega_{i}}, \alpha_j \rangle = \delta_{i, j}. \] 
  \begin{lemma}
  \label{sec2.4_lem1}
For any $i\in I$, the fundamental coweight $\check{\omega}_i$ is miniscule (cf.\,Definition \ref{Mini_coweight}) if and only if $i\in M$, if and only if $\langle \check{\omega}_i, \theta  \rangle=1$. 
  \end{lemma}
\begin{proof}
This can be read from the Kac labeling in \cite[p.54, Table Aff 1]{Ka} for $\tilde{\fg}$, and the list of miniscule fundamental weights for the dual Lie algebra $^L \fg$ of $\fg$ in \cite[p.72, ex.13]{Hu1}
\end{proof}

\begin{lemma} 
\label{sect2.4_lem2}
The set $\{ \check{\omega}_{\kappa} \, | \, \kappa \in \hat{M} \}$ gives a complete set of coset representatives of $\check{P}/\check{Q}$, as does the set $\{-\check{\omega}_{\kappa} \,|\, \kappa\in \hat{M}  \}$.
\end{lemma}

\begin{proof} 
We first show that all cosets $\check{Q}, \check{Q}+\check{\omega}_\kappa$, for  $\kappa\in M$ are all distinct. This is equivalent to showing that for any $\kappa\in M$, $\check{\omega}_\kappa\not\in \check{Q}$, and for any $\kappa, \kappa'\in M$, $\check{\omega}_\kappa-\check{\omega}_{\kappa'}\not \in \check{Q}$ if $\kappa\not=\kappa'$. This can be read directly from \cite[p.69,Table 1]{Hu1}. Hence $|\hat{M}|\leq  |\check{P}/\check{Q}|$.

By Lemma \ref{sec2.4_lem1}, $|M|$ is equal to the number of all miniscule fundamental coweights of $\fg$. On the other hand, $|\check{P}/\check{Q}|-1$ is equal to the number of all miniscule fundamental coweights of $\fg$ (cf.\cite[p.68; p.72,ex.13]{Hu1}). Therefore $|\check{P}/\check{Q}|= |\hat{M}|$. This completes the proof in the case of  $\{\check{\omega}_\kappa  \}$. The proof for $\{  -\check{\omega}_\kappa \}$ is the same.
\end{proof}

For any coweight $\mu \in \check{P}$, we define a conjugation ${\rm Ad}_{\mu}: W_{\sf{aff}}\to W_{\sf{aff}}$ given by 
\[   {\rm Ad}_{\mu}(\tau_\lambda w ):= \tau_{\mu}(  \tau_{\lambda}w)\tau_{-\mu}, \text{ for any } \tau_\lambda w\in W_{\sf{aff}}.\]  
For each $\kappa\in \hat{M}$, set $W_\kappa= {\rm Ad}_{\check{\omega}_\kappa }(W) $, where by convention $\check{\omega}_0=0$. Set $I_\kappa=\hat{I}\backslash \{\kappa\}$.
\begin{proposition}
\label{sect2.4_prop1}
For any $\kappa\in \hat{M}$, the subgroup  $W_{\kappa}$   is a parabolic subgroup of $W_{\sf{aff}}$ with Coxeter generators $ \{ s_i \li  i\in I_\kappa \} $.
\end{proposition}
\begin{proof}
When $\kappa=0$, $W_\kappa=W$ together with $\{s_i \li i\in I\}$ is clearly a Coxeter system.

Now we assume that $\kappa\in M$. 
For any $i\in \hat{I} \backslash \{ 0, \kappa \}$, it is easy to see that ${\rm Ad}_{\check{\omega}_\kappa } (s_i)=s_i$. Since $\theta=\sum_{i\in I } a_i \alpha_i$, by Lemma \ref{sec2.4_lem1} we have 
\[  {\rm Ad}_{\check{\omega}_\kappa }(s_{\theta} )=\tau_{\check{\theta}} s_\theta =s_0. \]
Thus, $W_\kappa$ contains $\{s_i \li  i\in I_\kappa  \}$. Let $W'_\kappa$ be the subgroup of $W_\kappa$ generated by $\{s_i \li  i\in I_\kappa \}$. As we see in the table \cite[p.54,Table Aff 1]{Ka}, by deleting $\kappa$ the Dynkin diagram $\hat{\Gamma}\backslash \{\kappa\}$ is the same as the Dynkin diagram $\Gamma$ of $\fg$. Therefore $W'_{\kappa}$ isomorphic to $W$ as Coxeter groups, in particular $| W'_\kappa|=|W|$. On the other hand the conjugation ${\rm Ad}_{\check{\omega}_\kappa }$ also gives  rise to an isomorphism $W\simeq W_\kappa$ of finite groups. It follows that $|W_{\kappa}|=|W'_{\kappa}|$. Therefore $W_\kappa=W'_\kappa$, and furthermore  $W_{\kappa}$   is a parabolic subgroup of $W_{\sf{aff}}$ with Coxeter generators $ \{ s_i \li  i\in I_\kappa \}$. 
\end{proof}
From  this proposition, we may deduce an interesting corollary on the finite Weyl group $W$. 
\begin{corollary}
The Weyl group $W$ can be generated by the set of reflections $\{s_i \li i\in I \backslash \{\kappa\} \}\cup \{s_\theta\}$ for any $\kappa\in M$ (equivalently,  the coefficient of $\alpha_\kappa$ in the highest root $\theta$ is 1). 
\end{corollary}
\begin{proof}
By Proposition \ref{sect2.4_prop1}, the set $\{ {\rm Ad}_{- \check{\omega}_\kappa }(s_i)  \li  i\in I_\kappa \}$ generates the Weyl group $W$. Unfolding the elements ${\rm Ad}_{- \check{\omega}_\kappa }(s_i) $, the corollary  immediately follows.
\end{proof}

The set $\pi_0(\Gr)$ of the components of the affine Grassmannian $\Gr$ can be identified with $\check{P}/\check{Q}$. Let $\Gr^{\kappa}$ be the component of $\Gr$ containing the point  $L_{-\check{ \omega}_\kappa}=t^{\check{-\omega_{\kappa}}} G(\mathscr{O})/G(\mathscr{O})$. By Lemma \ref{sect2.4_lem2}, we have the following disjoint union decomposition of $\Gr$
\[ \Gr=\sqcup_{\kappa \in \hat{M} }  \Gr^\kappa. \]
For any $\kappa\in \hat{M}$, set 
\begin{equation}  
\label{lattice_kappa}
\check{Q}_\kappa:= \check{Q}- \check{\omega}_\kappa. \end{equation}
Any $T$-fixed point in $\Gr^\kappa$ is given by  $L_{\lambda  }$  and any $\I$-orbit in $\Gr^\kappa$ is given by $\X_\lambda$, for some $\lambda\in \check{Q}_\kappa $. Note that there is a bijection $\iota_\kappa: \check{Q}_\kappa \to W_{\sf{aff}}/W_\kappa$ given by 
\[  \lambda\mapsto \tau_{-\lambda-\check{\omega}_\kappa}W_\kappa,\] 
which fits into the following commutative diagram
\[\begin{tikzcd}
\check{Q} \arrow{r}{ \iota_0 } \arrow[swap]{d}{\cdot - \check{\omega}_\kappa } & W_{\sf{aff}}/W  \arrow{d}{ {\rm Ad}_{\check{\omega}_\kappa }} \\
\check{Q}_\kappa  \arrow{r}{\iota_\kappa} & W_{\sf{aff}}/W_\kappa
\end{tikzcd}.
\]
In the following we would like to describe the component $\Gr^\kappa$ as a partial flag variety associated to the quotient $W_{\sf{aff}}/W_\kappa$.

There is a canonical projection $\pi: \tilde{G}_{\sf{sc}}\to G_{\sf{sc}}(\scr{K})\rtimes \bb{C}^\times $. The preimage
 \[  \tilde{B}=\pi^{-1}(\I_{\sf{sc}}\rtimes  \bb{C}^\times )\]
  is the standard Borel subgroup of  $\tilde{G}_{\sf{sc}}$, where $\I_{\sf{sc}}$ is the Iwahori subgroup in $G_{\sf{sc}}(\scr{K})$.
We have the following identification
\[ \tilde{G}_{\sf{sc}}/\tilde{B} \simeq  G_{\sf{sc} }(\scr{K})/  \I_{\sf{sc}}. \]
For any $\tau_\lambda w\in W_{\sf{aff}}$, we associate an $\I_{\sf{sc}}$-orbit 
\[   \Y_{\tau_\lambda w}:= \I_{\sf{sc}}t^{-\lambda }w\I_{\sf{sc}} /\I_{\sf{sc}}\subset  G_{\sf{sc} }(\scr{K})/  \I_{\sf{sc}}.\] 
Then $\dim Y_{\tau_\lambda w}=\ell(\tau_\lambda w)$. 
\begin{remark}
The sign normalization in $\Y_{\tau_\lambda w}$ is  crucial. Without this sign, the dimension formula for $\dim Y_{\tau_\lambda w}$ does not hold. 
\end{remark} 
  
    Let $\tilde{P}_\kappa$ denote the maximal standard parabolic subgroup of $\tilde{G}_{\sf{sc}}$ containing $\tilde{B}$, which is associated to the subset $I_\kappa= \hat{I}\backslash \{\kappa\}$ of $\hat{I}$. By Proposition \ref{sect2.4_prop1}, $\tilde{P}_\kappa$ corresponds to the parabolic subgroup $W_\kappa$ of $W_{\sf{aff}}$. The $\tilde{B}$-orbits in the partial flag variety $\tilde{G}_{\sf{sc}}/\tilde{P}_\kappa$ can be indexed by the cosets in $W_{\sf{aff}}/W_{\kappa}$.

Without confusion, we still denote by $p:  G_{\sf{sc}}(\bar{\scr{K} })\to G(\bar{\scr{K} }  )$  and $p:  G_{\sf{sc}}({\scr{K} })\to G({\scr{K} }  )$ the maps induced from the covering  map $p:G_{\sf{sc}}\to G$, where $\bar{\scr{K} }$ is the algebraic closure of $\scr{K}$. The map $p:  G_{\sf{sc}}(\bar{\scr{K} })\to G(\bar{\scr{K} }  )$ is surjective, but $p:  G_{\sf{sc}}({\scr{K} })\to G({\scr{K} }  )$ is not  if $\check{P}/\check{Q}\not=0$.

 Let $\tilde{t}_\kappa$ be a lifting of $t^{- \check{\omega}_\kappa}\in G(\scr{K})$ in $G_{\sf{sc}}(\bar{\scr{K} })$ via the map $p $. 
  Let ${\rm Ad}_{ \tilde{t}_\kappa }$ denote the conjugation map on $G_{\sf{sc}}(\bar{\scr{K} })$ by $ \tilde{t}_\kappa$, i.e.
  \[ {\rm Ad}_{ \tilde{t}_\kappa }(g):=\tilde{t}_\kappa g \tilde{t}_\kappa^{-1}, \text{ for any } g\in G_{\sf{sc}}(\bar{\scr{K} }). \]
 \begin{lemma}
 \label{sect2.4_lem3}
The conjugation ${\rm Ad}_{\tilde{t}_\kappa   }$ preserves $G_{\sf{sc}}(\mathscr{K})$, and  ${\rm Ad}_{\tilde{t}_\kappa   }$ is independent of the choice of the lifting $\tilde{t}_\kappa $.
\end{lemma}

\begin{proof} For any root $\alpha\in \Phi$, we have 
\begin{equation}
\label{lifting_cal}
\tilde{t}_\kappa  x_{ \alpha}(a) \tilde{t}_\kappa ^{-1}  =x_{\alpha}(at^{- \langle \check{\omega}_\kappa, \alpha \rangle}), 
\end{equation}
where $x_{\alpha}$ is a root subgroup homomorphism of $G_{\sf{sc}}$ associated to $\alpha$, and $a\in \scr{K}$. since the group $G_{\sf{sc}}(\mathscr{K})$ is generated by its root subgroups $x_{\alpha_i}(\mathscr{K})$ (cf.\,\cite[\S7,Theorem 10]{Stg}, it follows that the conjugation ${\rm Ad}_{\tilde{t}_\kappa   }$ preserves $G_{\sf{sc}}(\mathscr{K})$. The independence of the lifting $\tilde{t}_\kappa $ also follows from the formula (\ref{lifting_cal}).
\end{proof}
Set 
\[ G_{\sf{sc}}(\scr{O})_{\kappa}:={\rm Ad}_{  \tilde{t}_\kappa  }(G_{\sf{sc}}( \scr{O} )).\] 
\begin{lemma}
\label{sect2.4_lem4}
For any $\kappa\in \hat{M}$, the parabolic subgroup
$\tilde{P}_\kappa$ is equal to the preimage $\pi^{-1}(G_{\sf{sc}}(\scr{O})_{\kappa} \rtimes \bb{C}^\times  )$, where $\pi$ is the canonical projection $\pi: \tilde{G}_{\sf{sc}}\to G_{\sf{sc}}(\scr{K})\rtimes \bb{C}^\times $.
\end{lemma}
\begin{proof}
We choose root subgroup homomorphisms $\hat{x}_{\alpha_i}, \hat{x}_{-\alpha_i}: \bb{C}\to \tilde{G}_{\sf{sc}}$ associated to $\alpha_i$ for each $i\in \hat{I}$. For each $i\in I$, set $x_{\alpha_i}:=\pi\circ \hat{x}_{\alpha_i}$ and $x_{-\alpha_i}:=\pi\circ \hat{x}_{-\alpha_i}$. The images of $x_{\pm \alpha_i}$ land in $G_{\sf{sc}}$, and they are exactly  
the root subgroup homomorphisms of $G_{\sf{sc}}$ associated to $\pm\alpha_i$  for each $i\in I$. When $i=0$, there exist root subgroup homomorphisms $x_{\theta}, x_{-\theta}$ of $G_{\sf{sc}}$ associated to $\theta, -\theta$, such that 
\[  \pi \circ \hat{x}_{\alpha_0}(a)=x_{-\theta}(at), \quad  \pi \circ \hat{x}_{-\alpha_0}(a)=x_{\theta}(at^{-1}), \quad \text{ for any } a\in \bb{C}. \]
Set $\hat{P'_\kappa}:=\pi^{-1}(G_{\sf{sc}}(\scr{O})_{\kappa} \rtimes \bb{C}^\times  )$.
We first observe that  $\tilde{B}\subset \hat{P'_\kappa}$, since $\langle \check{\omega}_\kappa, \alpha\rangle\in \{0,1\}$ for any $\kappa\in \hat{M}$ and any positive root $\alpha\in \Phi^+$. Hence $\hat{P'_\kappa}$ is a standard parabolic subgroup of $\tilde{G}_{\sf{sc}}$ (cf.\,\cite[Theorem 5.1.3, Theorem 6.1.17]{Ku}). To show $\tilde{P}'_\kappa=\tilde{P}_\kappa$, it suffices to check that $\hat{P'_\kappa}$ corresponds to the subset $I_\kappa \subset \hat{I}$. In other words, it suffices to check that $\tilde{P}'_\kappa$ is proper and contains all root subgroups $\hat{x}_{- \alpha_i}$ for all $i\in I_\kappa$. Equivalently it suffices to show that $G(\scr{O})_\kappa$ contains $\pi\circ \hat{x}_{- \alpha_i}$ for all $i\in I_\kappa$.

When $\kappa=0$, it suffices to check that $G(\scr{O})$ contains $x_{- \alpha_i}$ for each $i\in I$, as well as $x_{-\theta}(\cdot t)$. This is obvious. 
When $\kappa\in M$, it suffices to show that $G_{\sf{sc}}(\scr{O})_{\kappa} $ is a proper subgroup of $G_{\sf{sc}}(\scr{K})$, and  $G_{\sf{sc}}(\scr{O})_{\kappa} $ contains $x_{- \alpha_i}$ for each $i\in I\backslash \{\kappa\}$, and $x_\theta(\cdot t^{-1})$. Clearly $G_{\sf{sc}}(\scr{O})_{\kappa} $ is a proper subgroup of $G_{\sf{sc}}(\scr{K})$, since $G_{\sf{sc}}(\scr{O})$ is proper in $G_{\sf{sc}}(\scr{K})$ and  by Lemma \ref{sect2.4_lem3} ${\rm Ad}_{\tilde{t}_\kappa }$ preserves $G_{\sf{sc}}(\scr{K})$. The group $G_{\sf{sc}}(\scr{O})_{\kappa} $ contains $x_{- \alpha_i}$ for each $i\in I\backslash \{\kappa\}$, since 
\[  \tilde{t}_\kappa  x_{- \alpha_i}(a) \tilde{t}_\kappa ^{-1}=x_{- \alpha_i}(a).
\]
Recall that $\langle \check{\omega}_\kappa, \theta \rangle=1$. By the  computation
  \[   \tilde{t}_\kappa  x_{\theta}(a) \tilde{t}_\kappa ^{-1}= x_{\theta}(at^{-1}),\] 
  we see that  $G_{\sf{sc}}(\scr{O})_{\kappa} $ contains  $x_{\theta}(a t^{-1})$ for any $a\in \mathbb{C}$. This completes the proof.
\end{proof}

\begin{example}
We examine $G=\mathtt{PGL}_3$. In this case $G_{\sf{sc}}=\mathtt{SL}_3$. All fundamental coweights are miniscule. Let $\check{\omega}_1$ be the first fundamental coweight. Then 
  \[t^{-\check{\omega}_1}=\begin{pmatrix} t^{-1} & 0 & 0 \\ 0 & 1 & 0 \\ 0 & 0 & 1 \\ \end{pmatrix} \in \mathtt{PGL}_3(\mathscr{K})\]
  has a lifting  
  \[\tilde{t}^{-\check{\omega}_1}=\begin{pmatrix} t^{-2/3} & 0 & 0 \\ 0 & t^{1/3} & 0 \\ 0 & 0 & t^{1/3} \\ \end{pmatrix} \in \mathtt{SL}_3(\bar{\scr{K} }).\] 
  By conjugating we see \[\mathtt{SL}_3(\scr{O} )_1:={\rm Ad}_{\tilde{t}^{-\check{\omega}_1}}(\mathtt{SL}_3(\mathscr{O}))= \Bigg\{  \begin{pmatrix} a_{11} & t^{-1}a_{12} & t^{-1}a_{12} \\ t a_{21} & a_{22} & a_{23} \\ ta_{31} & a_{32} & a_{33} \\ \end{pmatrix}\, \Bigg | \, (a_{ij})\in \mathtt{SL}_3(\scr{O}) \Bigg\}.
  \]  The group $p^{-1}( \mathtt{SL}_3(\scr{O} )_1\rtimes \bb{C}^\times )$ is the maximal parabolic subgroup of the Kac-Moody group $\widetilde{\mathtt{SL}_3}$ associated to the set of  simple roots $\{\alpha_0,\alpha_2\}$.
\end{example}

\begin{proposition}
\label{sect2.4_prop}
There exists an isomorphism $\mathfrak{i}_\kappa: \Gr^\kappa \to \tilde{G}_{\sf{sc}}/\tilde{P}_\kappa$ as $\tilde{G}_{\sf{sc}}$-homogeneous spaces with $\mathfrak{i}_\kappa(L_{-\check{\omega}_\kappa })=e\tilde{P}_\kappa$, and moreover, $\mathfrak{i}_\kappa( \X_\lambda   )= \tilde{B}\tau_{-\lambda -\check{\omega}_\kappa}\tilde{P}_\kappa/\tilde{P}_\kappa $.
\end{proposition}
\begin{proof}
It is well-known that the central component $\Gr^0$ is isomorphic to $G_{\sf{sc}}(\scr{K})/G_{\sf{sc}}(\scr{O})$ as a homogenous space (this can be seen by comparing the Cartan decompositions for $G(\scr{K})$ and $G_{\sf{sc}}(\scr{K})$). 

The translation map from $\Gr^0$ to $\Gr^\kappa$ given by $L\mapsto t^{-\check{\omega}_\kappa}L$ 
is an isomorphism between $\Gr^0$ and $\Gr^\kappa$. In view of Lemma \ref{sect2.4_lem3}, one can see that the component $\Gr^\kappa$ is also a homogeneous space of $G_{\sf{sc}}(\scr{K})$. Moreover $G_{\sf{sc}}(\scr{O})_\kappa$ is exactly 
the stabilizer group  of $G_{\sf{sc}}(\scr{K})$ at $L_{-\check{\omega}_\kappa }\in \Gr^\kappa$ (recall that $\tilde{t}_\kappa$ is a lifting of $t^{- \check{\omega}_\kappa}$ ). Hence $\Gr^\kappa\simeq G_{\sf{sc}}(\scr{K})/G_{\sf{sc}}(\scr{O})_\kappa$. Furthermore 
\[  \Gr^\kappa\simeq (G_{\sf{sc}}(\scr{K})\rtimes \bb{C}^\times) /(G_{\sf{sc}}(\scr{O})_\kappa\rtimes \bb{C}^\times), \]
since the rotation factor $\bb{C}^\times$ acts trivially at $L_{-\check{\omega}_\kappa }$. 

Finally we consider the action of $\tilde{G}_{\sf{sc}}$ on $\Gr^\kappa$, which factors through the action of $G_{\sf{sc}}(\scr{K})\rtimes \bb{C}^\times$. The stabilizer group of $\tilde{G}_{\sf{sc}}$ at $L_{-\check{\omega }_\kappa }$ is $p^{-1}(G_{\sf{sc}}(\scr{O})_{\kappa} \rtimes \bb{C}^\times  )$.
By Lemma \ref{sect2.4_lem4}, we conclude that $\Gr^\kappa\simeq \tilde{G}_{\sf{sc}}/\tilde{P}_{\kappa}$. 

For any $\lambda\in \check{Q}_\kappa$, equivalently $\lambda+ \check{\omega}_\kappa\in  \check{Q}$, so we have $t^{\lambda+ \check{\omega}_\kappa}\in G_{\sf{sc}}(\scr{K})$. Then
\[ \tilde{B}\tau_{-\lambda -\check{\omega}_\kappa}\tilde{P}_\kappa/\tilde{P}_\kappa\simeq \I_{\sf{sc}}t^{ \lambda+\check{\omega}_\kappa}\cdot L_{- \check{\omega}_\kappa}= \I_{\sf{sc}}\cdot L_{\lambda}=\I\cdot L_{\lambda}=\X_\lambda, \]
since the natural map $ \I_{\sf{sc}}\to \I$ is surjective. This completes the proof of this proposition. 

\end{proof}


\subsection{Proof of Theorem \ref{Thm_1}}
\label{sect2.5}
We first recall the statement of Theorem \ref{Thm_1}.
\begin{theorem*}
For any $\lambda\in \check{P}$, $\Psi(\lambda)=\Psi_\infty(\lambda)$.
\end{theorem*}
\begin{proof}
By Lemma \ref{sect_lem_1}, it suffices to show that  $\mu\in \Psi(\lambda)$ implies $\mu\in \Psi_\infty(\lambda)$. 
We may assume that $\lambda, \mu\in \check{Q}_\kappa$ for some $\kappa\in \hat{M}$. 

Let $\tau_{-\lambda-\check{\omega}_\kappa  }y$ (resp. $\tau_{-\mu-\check{\omega}_\kappa}w$) be the minimal representative in the coset $\tau_{-\lambda-\check{\omega}_\kappa  }W_\kappa$ (resp. $\tau_{-\mu-\check{\omega}_\kappa  }W_\kappa$). In view of Proposition \ref{sect2.4_prop} and \cite[Proposition 7.1.15]{Ku}, $\mu\prec_I \lambda$ is equivalent to $\tau_{-\mu-\check{\omega}_\kappa  }y\prec  \tau_{-\lambda-\check{\omega}_\kappa  }w$, where $\prec$ is the Bruhat order on $W_{\sf{aff}}$. 

By the chain property of partial Bruhat order for Coxeter groups (cf.\,\cite[Theorem 2.5.5]{BB}), there exists a sequence of elements 
\begin{equation}
\label{sequence_min}
\tau_{-\mu-\check{\omega}_\kappa}w= \tau_{-\lambda_n-\check{\omega}_\kappa}y_n \prec \tau_{-\lambda_{n-1} -\check{\omega}_\kappa}y_{n-1}\prec \cdots \prec   \tau_{-\lambda_1-\check{\omega}_\kappa}y_1\prec   \tau_{-\lambda_0-\check{\omega}_\kappa}y_0= \tau_{-\lambda-\check{\omega}_\kappa}y, \end{equation}
which satisfies the following properties:
\begin{enumerate}
\item for each $0\leq i\leq n$, $ \tau_{-\lambda_i-\check{\omega}_\kappa}y_i$ is the minimal representative in the coset $ \tau_{-\lambda_i-\check{\omega}_\kappa}W_\kappa$;
\item for each $1\leq i\leq n$, $ \tau_{-\lambda_i-\check{\omega}_\kappa}y_i=s_{\beta_i,k_i}  \tau_{-\lambda_{i-1}-\check{\omega}_\kappa}y_{i-1} $ for some affine reflection $s_{\beta_i,k_i} $ with $\beta_i\in \Phi^+$ and for some $k_i\in \mathbb{Z}$, and $\ell(  \tau_{-\lambda_i-\check{\omega}_\kappa}y_i)=\ell( \tau_{-\lambda_{i-1}-\check{\omega}_\kappa }y_{i-1})-1$. 
\end{enumerate}
By the choice of this sequence, we see that all $\lambda_i$ are distinct. 
By Proposition \ref{sect2.4_prop} and \cite[Proposition 7.1.15]{Ku} again, 
the sequence (\ref{sequence_min}) is equivalent to 
\[ \lambda_n=\mu\prec_I  \lambda_{n-1}\prec_I \cdots \prec_I \lambda_{i} \cdots\prec_I  \lambda_1\prec_I  \lambda_0=\lambda. \] 

From the following computation
\[ \tau_{-\lambda_i-\check{\omega}_\kappa }y_i=\tau_{-k_i\check{\beta}_{i}  }s_{\beta_i}   \tau_{-\lambda_{i-1} -\check{\omega}_\kappa}y_{i-1} =\tau_{-\lambda_{i-1} +(\langle \lambda_{i-1}, \beta_i \rangle -k_i ) \check{\beta}_i -\check{\omega}_\kappa} {\rm Ad}_{\check{\omega}_\kappa }(s_{\beta_i}) y_{i-1}, \]
we see that $\lambda_i=\lambda_{i-1}-   (\langle \lambda_{i-1}, \beta_i \rangle -k_i ) \check{\beta}_i $ and $y_i= {\rm Ad}_{\check{\omega}_\kappa }(s_{\beta_i}) y_{i-1} \in W_\kappa$. It follows that $\langle \lambda_{i-1}, \beta_i \rangle -k_i\not=0$, since all $\lambda_i$ are distinct. 
Since 
\[\tau_{-\lambda_{i-1}-\check{\omega}_\kappa }y_{i-1} \prec \tau_{-\lambda_i-\check{\omega}_\kappa}y_i= s_{\beta_i,k_i}  \tau_{-\lambda_{i-1}-\check{\omega}_\kappa}y_{i-1},\]
 in view of Proposition \ref{sect2.3_prop1} and Corollary \ref{sect2.3_cor1}, we have
\begin{enumerate}
\item if $k\geq 0$, then $0< \langle \lambda_{i-1}, \beta_i \rangle -k_i\leq   \langle \lambda_{i-1}, \beta_i \rangle$;
\item if $k<0$, then $0< k_i- \langle \lambda_{i-1}, \beta_i \rangle<- \langle \lambda_{i-1}, \beta_i \rangle$. 
\end{enumerate}
Therefore $\lambda_i\in  \Psi_1(\lambda_{i-1})$ for any $1\leq i\leq n$ (cf.\,Definition \ref{sect2.2_def1}). In fact $\lambda_i\in \Psi_i(\lambda_0)$ for each $i$.
 It follows that $\mu\in \Psi_\infty(\lambda)$, since by convention $\mu=\lambda_n$ and $\lambda=\lambda_0$.
\end{proof}

\subsection{Crossing different components of $\Gr$}
\label{sect2.6}
Recall that $\tilde{W}_{ \sf{aff}}$ is the extended affine Weyl group $\check{P}\rtimes W$. Let $\Omega$ denote the group of all length zero elements in $\tilde{W}_{ \sf{aff}}$. Equivalently, $\Omega$ is the stabilizer group of $\tilde{W}_{ \sf{aff}}$ at the fundamental alcove of $W_{\sf{aff}}$ (cf.\,\cite[\S 4.5]{Hu2}). Then $\Omega\simeq \check{P}/\check{Q}$. 
Let $W_{\check{\omega}_\kappa }$ denote the stabilizer group of $W$ at $\check{\omega}_\kappa  $. Then $W_{\check{\omega}_\kappa }$ is a parabolic subgroup of $W$ with Coxeter generators $\{s_i \li  i\in \hat{ I }\backslash \{0, \kappa\} \}$. 
For each $\kappa\in \hat{M}$, 
let $w_\kappa$ denote the longest element in $W_{\check{\omega}_\kappa }$ for each $\kappa\in \hat{M}$, in particular $w_0$ is the longest element in $W$. For each $\kappa\in \hat{M}$, set $w^{\kappa}=w_\kappa w_0$.
\begin{lemma}
The group $\Omega$ consists of elements 
\[ \gamma_\kappa:= \tau_{\check{ \omega }_\kappa } w^\kappa\in \tilde{W}_{\sf{aff}},\, \kappa\in \hat{M},\]
\end{lemma}
\begin{proof}
We first make the following observations, for any $\kappa\in {M}$,
\begin{enumerate}
\item  for any $\beta\in \Phi^+$, if the support of $\beta$ contains the simple root $\alpha_\kappa$, then the coefficient must be $1$. This follows from the assumption that $\check{\omega}_\kappa$ is miniscule (cf.\,Definition \ref{Mini_coweight}).
\item  for any  $\beta\in \Phi^+$, $w_\kappa(\beta)\in \Phi^+$ if and only if $\beta$ contains the simple root $\alpha_\kappa$. This follows from the fact that $w_\kappa$ maps all simple roots in $\Phi^+$ except $\alpha_\kappa$ to negative roots.
\end{enumerate}
Based on these two observations, it is now easy to see that by the formula (\ref{length_formula}), for any $\kappa\in \hat{M}$, $\gamma_\kappa$ is of length zero. On the other hand, since $|\Omega|=|\check{P}/\check{Q}|=|\check{M}|$ (cf.\,Lemma \ref{sect2.4_lem2}), these length zero elements exhaust all elements of $\Omega$.
\end{proof}
 
 Using the length zero element $\gamma_\kappa$, we define a translation map $\rho_\kappa:  \Gr^0\to  \Gr^\kappa $  given by 
 \[  L \mapsto     t^{-\check{\omega}_\kappa }w^\kappa L \in  \Gr^\kappa, \text{ for any }L\in \Gr^0. \]
 \begin{proposition}
 \label{Comp_Tran}
 For any $\lambda\in \check{Q}$, we have  $\rho_\kappa(\X_\lambda)= \X_{ w_\kappa(\lambda)-\check{\omega}_\kappa }$, and $\rho_\kappa(\overline{\X}_\lambda)= \overline{\X}_{ w_\kappa(\lambda)-\check{\omega}_\kappa }$
 \end{proposition}
 \begin{proof}
 Since the map $\rho_\kappa$ is an isomorphism between $\Gr^0$ and $\Gr^\kappa$, it suffices to show that  $\rho_\kappa(\X_\lambda)= \X_{ w_\kappa(\lambda)-\check{\omega}_\kappa }$. We are reduced to show the following fact  
 \[ t^{-\check{\omega}_\kappa }w^\kappa \I  (w^\kappa)^{-1} t^{\check{\omega}_\kappa }= \I. \]
 For any $\alpha\in \Phi^+$, 
 \[ t^{-\check{\omega}_\kappa }w^\kappa x_\alpha(a)(w^\kappa)^{-1} t^{\check{\omega}_\kappa }=x_{ w^\kappa(\alpha) } (a t^{- \langle \check{\omega}_\kappa, w^\kappa(\alpha) \rangle   }), \]
 where $x_\alpha$ (resp. $x_{ w^\kappa(\alpha) }$) is the root group homomorphism associated to $\alpha$ (resp. $w^\kappa(\alpha)$), and $a\in \bb{C}$. 
  Notice that $w^\kappa( \alpha) \in \Phi^- $ if and only if the support of $w^\kappa( \alpha)$ contains the simple root $\alpha_\kappa$ with coefficient $-1$. It follows that 
 \[  - \langle \check{\omega}_\kappa, w^\kappa(\alpha) \rangle     =   \begin{cases}  1    \quad  \text{ if } w^\kappa( \alpha) \in \Phi^- \\       0      \quad  \text{ if } w^\kappa( \alpha) \in \Phi^+   \end{cases}. \]
  It follows that 
  \[t^{-\check{\omega}_\kappa }w^\kappa x_\alpha(a)(w^\kappa)^{-1} t^{\check{\omega}_\kappa } \in  \I.\]
 For any $\alpha\in \Phi^-$, 
  \[ t^{-\check{\omega}_\kappa }w^\kappa x_\alpha(at)(w^\kappa)^{-1} t^{\check{\omega}_\kappa }=x_{ w^\kappa(\alpha) } (a t^{1- \langle \check{\omega}_\kappa, w^\kappa( \alpha) \rangle  }). \]
In this case $w^\kappa( \alpha) \in \Phi^+ $ if and only if  the support of $w^\kappa( \alpha)$ contains the simple root $\alpha_\kappa$ with coefficient $1$. It follows that 
 \[ 1 - \langle \check{\omega}_\kappa, w^\kappa(\alpha) \rangle     =   \begin{cases}  0    \quad  \text{ if } w^\kappa( \alpha) \in \Phi^+ \\       1    \quad  \text{ if } w^\kappa( \alpha) \in \Phi^-   \end{cases}. \]
Therefore 
\[t^{-\check{\omega}_\kappa }w^\kappa x_\alpha(at)(w^\kappa)^{-1} t^{\check{\omega}_\kappa } \in  \I.\]
This finishes the proof. 
 \end{proof}

Now we define a map $\bar{\rho}_\kappa:  \check{Q}\to  \check{Q}_\kappa $, given by $\bar{\rho}_\kappa(\lambda)= w^\kappa(\lambda)-\check{\omega}_\kappa$. The following corollary is a consequence of Proposition \ref{Comp_Tran}.

\begin{corollary}
\label{sect2.6_cor}
The map $\bar{\rho}_\kappa$ gives rise to a bijection from  $\Psi(\lambda)$ to $\Psi(w^\kappa(\lambda) -\check{\omega}_\kappa )$ which preserves the partial order $\prec_I$, for any $\lambda\in \check{Q}$ and $\kappa\in \hat{M}$.
\end{corollary}
Notice that for any $\alpha\in \Phi^+$, if $\langle\lambda, \alpha  \rangle>0 $ and $0\leq k \leq  \langle \lambda, \alpha   \rangle $, by Corollary \ref{sect2.6_cor}
\[\bar{\rho}_\kappa( \lambda-k\check{\alpha} )= \bar{\rho}_\kappa(\lambda) - k w^\kappa(\check{\alpha}) \in \Psi\big(w^\kappa(\lambda) -\check{\omega}_\kappa \big).\]
Given a positive root $\alpha$ satisfying above conditions, $w^\kappa(\alpha)$ could be a negative root. If  $w^\kappa(\alpha)\in \Phi^-$, then 
\[ \langle \bar{\rho}_\kappa(\lambda), -w^\kappa(\alpha)\rangle< 0.\] In this case, 
$\lambda-k\check{\alpha}$ is obtained by successively subtracting the positive coroot $\check{\alpha}$, however $\bar{\rho}_\kappa( \lambda-k\alpha )$ is obtained by successively adding the positive coroot $-w^\kappa(\check{\alpha})$. 
 There is a similar phenomenon when $\langle\lambda, \alpha  \rangle<0 $.

\section{A duality between affine Schubert varieties and level one affine Demazure modules}

\subsection{The Kac-Moody algebra $^L \tilde{\fg}$}
\label{sect_duality1}
Let $^L \hat{\Gamma}$ be the Dynkin diagram which is dual to the Dynkin diagram $\hat{\Gamma}$ of $\tilde{\fg}$. Let $^L \tilde{\fg}$ denote the Kac-Moody algebra associated to $^L\hat{\Gamma}$. 
Let $\{\check{e}_i,\check{f}_i \li, i\in \hat{I} \}$ be a set of Chevalley generators of $\Lg$. Let $^L\fg$ be the Lie subalgebra of $\Lg$ generated by $\{\check{e}_i, \check{f}_i \li  i\in I  \}$. Then $^L\fg$ is a simple Lie algebra with Dynkin diagram $^L\Gamma$ which is dual to the Dynkin diagram $\Gamma$ of $\fg$. 
We have the following table for the correspondence between $(\Gamma, \hat{\Gamma})$ and $({^L}\Gamma, {^L} \hat{\Gamma})$
\begin{equation}
\label{Table_Dual_diagram_affine}
\renewcommand{\arraystretch}{1.5}
 \begin{tabular}{|c?c  | c | c |c |c |c | c| c| c|c |c|c|c|c|c |c ||} 
 \hline  
 $\Gamma$ &   $A_n$  & $B_n$ &  $C_n$ &  $D_n$  &   $E_6$  &   $E_7$   &  $E_8 $  & $ F_4 $  & $ G_2$     \\ [0.5ex] 
  \hline
$\hat{\Gamma}$ &   $A_n^{(1)}$  & $B_n^{(1)}$ &  $C_n^{(1)}$ &  $D_n^{(1)}$  &   $E_6^{(1)}$  &   $E_7^{(1)}$   &  $E_8^{(1)} $  & $ F_4^{(1)} $  & $ G_2^{(1)} $     \\ [0.5ex] 
\specialrule{.1em}{.05em}{.05em} 
$^L \Gamma$ & $A_n$  &  $C_n$ & $B_n$  & $D_n$ &  $E_6$  & $E_7$  &  $E_8 $  & $F_4$ & $G_2$   \\  [1ex] 
  \hline
$^L \hat{\Gamma}$ & $A_n^{(1)}$  &  $A^{(2)}_{2n-1}$ & $D^{(2)}_{n+1}$  & $D_n^{(1)}$ &  $E_6^{(1)}$  & $E_7^{(1)}$  &  $E_8^{(1)} $  & $E_6^{(2)}$ & $D_4^{(3)}$   \\  [0.5ex] 
 \hline
\end{tabular}.
\renewcommand{\arraystretch}{1.5}
\end{equation}
From this table, we see that if $\Gamma$ is simply-laced, then $\hat{\Gamma}=   {^L}\hat{\Gamma}$. If $\Gamma$ is non simply-laced, then $^L \hat{\Gamma}$ is of twisted affine type.

Let $^L \tilde{\fh}$ denote the Cartan subalgebra of $\Lg$. We can write 
\[  \Lh={^L}\fh \oplus \bb{C}\check{c}\oplus \bb{C} \check{d}, \text{ and } \Lh^*={^L}\fh^* \oplus \bb{C}\check{\Lambda}_0 \oplus \bb{C} \check{\delta},\]
where ${^L}\fh$ is the Cartan subalgebra of $^L\fg$, $\check{c}$ is the canonical center of $\Lg$, $\check{d}$ is the degree operator, $\check{\Lambda}_0$ is the fundamental weight of $\Lg$ associated to $0\in \hat{I}$, and $\check{\delta}$ is the element such that $\check{\delta}|_{\Lh }=0$, $\langle \check{\delta}, \check{c} \rangle=0$ and $\langle \check{\delta}, \check{d} \rangle=1$. 
 
Recall that $\tilde{\fh}$ is the Cartan subalgebra of $\tilde{\fg}$.
Under the duality of $\hat{\Gamma}$ and $^L\hat{\Gamma}$, we may identify $\Lh$ with $\tilde{\fh}^*$ and identify $\Lh^*$ with $\tilde{\fh}$. Under this identification, $^L\fh=\fh^*$ and ${^L}\fh^*=\fh$. In particular, the simple roots $\{\alpha_i\li  i\in \hat{I} \}$ of $\tilde{\fg}$ can be regarded as the simple coroots of $\Lg$, and the simple coroots $\{\check{\alpha}_i \li i\in \hat{I}\}$ of $\tilde{\fg}$ can be regarded as simple roots of $\Lg$. Moreover, $\delta$ can be regarded as the canonical center $\check{c}$, hence
\begin{equation}
\label{center_dual}
\check{c}=\sum_{i\in \hat{I}} a_i \alpha_i, \end{equation}
where $a_i$ is the Kac labeling of $\hat{\Gamma}$ at $i$, in particular $a_0=1$. The coroot lattice $\check{Q}$ (resp. coweight lattice $\check{P}$) are now regarded as root lattice (resp. weight lattice) of $^L\fg$. 

We will still denote by $(\cdot |\cdot )$ the induced bilinear forms on $\Lh$ and $\Lh^*$ from the normalized bilinear form $(\cdot|\cdot )$ on $\tilde{\fh}$ and $\tilde{\fh}$. It turns out that the induced forms on  $\Lh$ and $\Lh^*$ are still the normalized bilinear forms with respect to $\Lg$. 

Recall the affine Weyl group $W_{\sf{aff}}=\check{Q}\rtimes W$. We can also realize $W_{\sf{aff}}$ as the Weyl group of $\Lg$. For any $\tau_\lambda w\in W_{\sf{aff}}$ and $h \in \Lh^*$, following \cite[\S6.5.5]{Ka} we define 
\begin{equation}
\label{formula_affine_2}
 \tau_{\lambda}w(h)=h+\langle h, \check{c}  \rangle \lambda - ( (h| \lambda) + \frac{ (\lambda|\lambda) }{2} \langle h, \check{c} \rangle   )\check{\delta}. 
 \end{equation}

 The set  ${^L}\hat{\Phi}_{\sf{re}}$ of real roots of $\Lg$  can be described as follows (cf.\cite[Prop.6.3 a)b) ]{Ka})
 \[  {^L}\hat{\Phi}_{\sf{re}}= \{ \check{\alpha}+ k r_\alpha \check{\delta} \li  \alpha \in   \Phi, k\in \bb{Z}  \}   
  \] 
 where $\check{\alpha}$ is the coroot associated to $\alpha$, and $r_\alpha=\frac{2}{(\alpha| \alpha)}$, more precisely
 \[r_\alpha:=\begin{cases}   1   \quad   \text{ if  $\Gamma$ is simply-laced, or  } \alpha \text{ is a long root when $\Gamma$ is non simply-laced} \\  
    2   \quad  \text{ if  $\alpha$ is a short root when } \Gamma= B_n,C_n,F_4 \\   
       3 \quad  \text{ if  $\alpha$ is a short root when } \Gamma=G_2    \end{cases}.\]
There is a bijection $\eta:\hat{\Phi}_{\sf{re}}\to {^L}\hat{\Phi}_{\rm re}$ between $\hat{\Phi}_{\rm re}$ and ${^L}\hat{\Phi}_{\sf{re}}$, given by 
\[   \eta(\alpha+ k\delta)=  \check{\alpha}+ kr_\alpha \check{\delta }.\]
     \begin{lemma}
     \label{eta_equivariant}
    The bijection $\eta$ is $W_{\sf{aff}}$-equivariant. 
     \end{lemma}
     \begin{proof}
     For any $\tau_\lambda w\in W_{\sf{aff}}$ and $\alpha+k\delta\in \hat{\Phi}$, from the formula (\ref{formula_affine}) we have
     \begin{align*}
        \eta(\tau_\lambda w(\alpha+k\delta )  )&= \eta ( w( \check{\alpha} )+ (k-    \langle \lambda, w(\alpha)  \rangle  ) \delta  ) \\
                                                        &=w(\check{\alpha})+ (k-    \langle \lambda, w(\alpha)  \rangle  ) r_\alpha \check{\delta},
        \end{align*}
        where $ w(\check{\alpha})$ is equal to the coroot of $w(\alpha)$. On the other hand, 
\begin{align*}
 \tau_\lambda w (\eta( \alpha+k\delta ) ) &=  \tau_\lambda w ( \check{\alpha} +k r_\alpha \check{\delta} ) \\
                                                               & = w(\check{\alpha} ) +   ( kr_\alpha - (  \lambda | w(\check{ \alpha} )  )     ) \check{\delta}    \\
                                                               &= w(\check{\alpha} ) + ( kr_\alpha -  \langle \lambda, w(\check{ \alpha} )  \rangle r_\alpha   ) \check{\delta},
 \end{align*}
 where the second equality follows from (\ref{formula_affine_2}), and the third equality holds since
 \[ (  \lambda | w(\check{ \alpha} )  )   = \frac{ 2    \langle \lambda, w(\check{ \alpha} )  \rangle  }{  ( w(\alpha) | w( \alpha ) ) }= \frac{ 2    \langle \lambda, w(\check{ \alpha} )  \rangle  }{   ( \alpha |   \alpha ) }=   \langle \lambda, w(\check{ \alpha} )  \rangle r_\alpha. \]
                                                             From the above two computations, we see that $\eta$ is $W_{\sf{aff}}$-equivariant. 
                                                               \end{proof}
Under the map $\eta$, the image $\eta(\hat{\Phi}_{\sf{re}}^\pm )$ is the set of all positive (resp. negative) roots in ${^L}\hat{\Phi}_{\rm re}$.
Recall the set  $\{ \check{\omega}_\kappa \li  \kappa\in \hat{M} \}$ in Lemma \ref{sect2.4_lem2}. The following lemma follows from the discussion in \cite[\S 12.4]{Ka}. For the convenience of the reader, we include an argument here. 
\begin{lemma}
\label{level_one_weight_lem}
For any $\lambda\in \check{P}^+$, $\check{\Lambda}_0+ \lambda$ is a dominant weight of $\Lg$ of level one if and only if $\lambda \in \{ \check{\omega}_\kappa \li  \kappa\in \hat{M}  \}$.
\end{lemma}
\begin{proof}
We first observe that  $\lambda$ is dominant if and only if 
\[ \langle\check{\Lambda}_0+ \lambda, \alpha_i  \rangle= \langle \lambda, \alpha_i  \rangle \geq 0, \quad \text{ for any } i\in I. \]
By the formulae (\ref{center_dual}) and (\ref{theta_formula}), $\alpha_0=\check{c}-\theta$. Hence, $\check{\Lambda}_0+ \lambda$ is dominant if and only 
\[  \langle\check{\Lambda}_0+ \lambda, \alpha_0  \rangle=  \langle\check{\Lambda}_0+ \lambda, \check{c}-\theta  \rangle=1- \langle \lambda, \theta  \rangle\geq 0. \]
By Lemma \ref{sec2.4_lem1}, $\lambda$ is equal to $\check{\omega}_\kappa$ for $\kappa\in \hat{M}$. This concludes the proof.

\end{proof}
For each $\kappa\in \hat{M}$, set $\check{\Lambda}_\kappa:= \check{\Lambda}_0+ \check{\omega}_\kappa $. 
\begin{lemma}
\label{stabilizer_fundamental_wt}
The stabilizer group of $W_{\sf{aff}}$ at the dominant weight $\check{\Lambda}_\kappa$ is equal to $W_\kappa$.
\end{lemma}
\begin{proof}
It is known that the stabilizer $W''_\kappa$ of $W_{\sf{aff}}$ at $\check{\Lambda}_\kappa$ is a parabolic subgroup of $W_{\sf{aff}}$. It is enough to determine the Coxeter generators of $W''_\kappa$. 
We first examine $s_0( \check{\Lambda}_0 )$,
\[ s_0( \check{\Lambda}_0 )=\tau_{\check{\theta}} s_\theta  (  \check{\Lambda}_0  )=  \tau_{\check{\theta}  } (  \check{\Lambda}_0  ) =  \check{\Lambda}_0 +\check{\theta} -\frac{ (\check{\theta}| \check{\theta}   )  }{  2} \check{\delta}=  \check{\Lambda}_0 +\check{\theta} - \check{\delta}.\]
For any $\kappa\in M$, 
\begin{align*}
 s_0( \check{\Lambda}_\kappa )&=\tau_{\check{\theta}} s_\theta  (  \check{\Lambda}_0 +\check{\omega}_\kappa  ) 
                                                   =    \tau_{\check{\theta}}(   \check{\Lambda}_0 +\check{\omega}_\kappa  -\check{\theta})  \\
                                                   &=  \check{\Lambda}_0 +\check{\omega}_\kappa  -\check{\theta}+  \langle \check{\Lambda}_0 +\check{\omega}_\kappa -\check{\theta}, \check{c}  \rangle \check{\theta} -(  (\check{\theta} |   \check{\Lambda}_0 +\check{\omega}_\kappa  -\check{\theta}  )+  \frac{ (\check{\theta }| \check{\theta}  ) }{2}  \langle \check{\Lambda}_0 +\check{\omega}_\kappa, \check{c}  \rangle       ) \check{\delta}\\
                                                   &=\check{\Lambda}_0 +\check{\omega}_\kappa -\big( (\check{\theta} |  \check{\omega}_\kappa  -\check{\theta}  )+  \frac{ (\check{\theta }| \check{\theta}  ) }{2}     \big) \check{\delta}
                                                   =\check{\Lambda}_\kappa,
                                                   \end{align*}
                                                   where the  fourth equality holds since $(\check{\theta }| \check{\theta}  )  =2 $  and $(\check{\theta} |  \check{\omega}_\kappa  -\check{\theta}  )=-1$. 
                                                   
           Now it is easy to see that the Coxeter generators of $W''_\kappa$ are given by $\{  s_i \li  i\in I_\kappa \}$, where $I_\kappa= \hat{I}\backslash  \{ \kappa \}$. By Proposition \ref{sect2.4_prop1}, we can conclude that   $W''_\kappa=W_\kappa$. 

\end{proof}

\subsection{Level one affine Demazure Modules}
\label{Sect_affine_Dem}
We first make a digression to prove a general lemma in the setting of general symmetrizable Kac-Moody algebras. 
Let $\Lambda$ be an integral dominant weight of a  symmetrizable Kac-Moody algebra $\mathcal{G}$ with a fixed Borel subalgebra $\cal{B}$. Let $\mathcal{W}$ be the Weyl group of $\cal{G}$. Let $V_{\Lambda}$ denote the irreducible integrable representation of $\cal{G}$ of highest weight $\Lambda$. Let $\mathcal{W}_\Lambda$ denote the stabilizer group of $\cal{W}$ at $\Lambda$. 
Then $\cal{W}_\Lambda$ is a parabolic subgroup of $\cal{W}$. Let $\cal{W}^{\Lambda}$ denote the set of minimal coset representatives in $\cal{W}/ \cal{W}_{\Lambda}$. 
\begin{proposition}
\label{sect4.1_lem}
For any $ {y}, {w}\in \cal{W}^{\Lambda}$, we have 
 \[ {w} \prec {y} \iff v_{{w} (\Lambda)  } \in U( \cal{B}   ) v_{{y}(\Lambda)   }   \]
\end{proposition}
\begin{proof}
Let $\cal{B}^-$ denote the negative Borel subalgebra in $\cal{G}$. Then we have 
  \begin{equation}
  \label{Kumar_parabolic}
         w\prec y  \iff   v_{y(\Lambda)}\subset   U( \cal{B}^-  ) v_{w(\Lambda)   }. \end{equation}
This fact is stated in  \cite[ex.7.1.E.4 ]{Ku}, and this is a parabolic version  of \cite[Prop. 7.1.20]{Ku}. The proof is almost identical to that in \cite[Prop. 7.1.20]{Ku}. 

In the next step, we show that 
\[  v_{y(\Lambda)}\subset   U( \cal{B}^-  ) v_{w(\Lambda)   }    \iff     v_{w(\Lambda)}\subset   U( \cal{B}  ) v_{y(\Lambda)   }. \]
Let  $(\cdot,\cdot )$ be the contravariant form on $\scr{H}_\Lambda$ (cf.\,\cite[Prop. 2.3.2]{Ku}), which satisfies that, for any $v_1,v_2\in V_\Lambda$, and $x\in \cal{G}$.
\[ (x v_1, v_2)= (v_1, \sigma(x)  v_2), \] 
where $\sigma: \cal{G}\to \cal{G}$ is the Cartan involution on $\cal{G}$. 
 The involution $\sigma$ induces an anti-automorphism on the universal enveloping algebra $U(\cal{G})$ of $\cal{G}$. Furthermore the contravariant form $(\cdot,\cdot )$ has the following properties:
 \begin{enumerate}
 \item  $(\cdot,\cdot )$ is non-degenerate on each weight space $V_\Lambda(\mu)$, where $\mu$ is a weight of $\cal{G}$.
 \item  $(v_1,v_2)=0$ for any two weight vectors $v_1,v_2$ of distinct weights.
 \end{enumerate}

Assume that $v_{y(\Lambda)} \subset   U( \cal{B}^-  ) v_{w(\Lambda)   } $. We may write $v_{y(\Lambda)} =P v_{w(\Lambda) }$, where $P$ is an element in the enveloping algebra  $U(\cal{N}^-) $ of the nilpotent radical $\cal{N}^-$ of $\cal{B}^-$. 
Notice that  the weight space $V_{\Lambda}( y(\Lambda) )$ is 1-dimensional. It follows that we may assume $P\in  U(\cal{N}^-)  $ is a monomial in negative root vectors in $\cal{N}^-$. By the non-degeneracy of $(\cdot,\cdot )$ on $V_{\Lambda}( y(\Lambda) )$, we have 
\[ 0\not=(v_{y(\Lambda)}, v_{y(\Lambda)}   )= (P v_{w(\Lambda)}, v_{y(\Lambda)}) =( v_{w(\Lambda)}, \sigma(P) v_{y(\Lambda)} ). \]
Note that $\sigma(P)\in U(\cal{N})$ is a monomial in positive root vectors in $\cal{N}$, where $\cal{N}$ is the unipotent radical of $\cal{B}$. 
By the second property of the contravariant form mentioned above, we must have $ \sigma(P) v_{y(\Lambda)}\in V_{\Lambda}( w(\Lambda) )$. By the one-dimensionality of $V_{\Lambda}( w(\Lambda) )$, there exists a nonzero constant $c$ such that 
\[ v_{w(\Lambda)} =c\sigma(P)  v_{y(\Lambda)} \in  U(\cal{N})   v_{y(\Lambda)}. \]
By similar argument, we can show that if $ v_{w(\Lambda)}\in   U( \cal{B}  ) v_{y(\Lambda)   }$, then $v_{y(\Lambda)} \in   U( \cal{B}^-  ) v_{w(\Lambda)   } $. Thus the lemma is proven.
\end{proof}
\begin{remark}
The equivalence (\ref{Kumar_parabolic}) can be proved by the induction on the length of elements of Weyl group. The proof of \cite[Prop. 7.1.20]{Ku} in the regular case essentially incorporates the original proof for finite Weyl groups by Bernstein-Gelfand-Gelfand \cite{BGG}. We don't know how to prove Proposition \ref{sect4.1_lem} by induction directly. It is interesting to use the contravariant form to reduce the lemma to the equivalence (\ref{Kumar_parabolic}).
\end{remark}


We now return to the Kac-Moody algebra $\Lg$.
Let $\scr{H}_\kappa$ denote the irreducible integrable representation of $\Lg$ of highest weight $\check{\Lambda}_\kappa$. Then $\scr{H}_\kappa$ is an integrable representation of level one. 
Fix a highest weight vector $v_{\check{\Lambda}_\kappa} \in \scr{H}_\kappa$. For any $\lambda \in \check{P}$ if  $\lambda \in \check{Q}_\kappa:= \check{Q}-\check{\omega}_\kappa$, 
then we set
\[ \varpi(\lambda):=  \tau_{-\lambda - \check{\omega}_\kappa }( \check{\Lambda}_\kappa    ),\]
where $ \tau_{-\lambda - \check{\omega}_\kappa }\in W_{\sf{aff}}$. Let $v_{ \varpi(\lambda) }\in    \scr{H}_\kappa$ be
 an extremal vector of weight $\varpi(\lambda)$. By this convention, $\varpi(-\check{\omega}_\kappa )=\check{\Lambda}_\kappa$. 
By a simple computation from formula (\ref{formula_affine_2}), we have the following formula
\begin{equation}\label{maximal_wt_formula}
  \varpi(\lambda)=\check{\Lambda}_0-\lambda-\frac{(\lambda | \lambda)- (\check{\omega}_\kappa| \check{\omega}_\kappa ) }{2} \check{\delta}, \quad \text{ for } \lambda\in \check{Q}_\kappa. \end{equation}

\begin{definition}
We now define the affine Demazure module $\hat{\scr{D} }_\lambda $ for each $\lambda\in \check{P}$ as follows,
\[   \hat{\scr{D} }_\lambda :=U(  \Lb  )   v_{ \varpi( \lambda ) } \subset   \scr{H}_\kappa, \quad  \text{ if } \lambda\in \check{Q}_\kappa, \]
where $U(  \Lb  )$ is the universal enveloping algebra of the  Borel subalgebra $\Lb $ of $\Lg$.
\end{definition}

Recall that there is a bijection $  \check{Q}_\kappa  \simeq  W_{\sf{aff}}/W_{\kappa}$ where $\check{Q}_\kappa:= \check{Q}- \check{\omega}_\kappa$.
For each $\lambda\in \check{Q}_\kappa $, let $\bar{\tau}_{-\lambda - \check{\omega}_\kappa }$ denote the associated minimal representative in the coset $\tau_{-\lambda-\check{\omega}_\kappa }W_\kappa$.

\begin{proposition}
\label{Bruhat_Dem_affine_prop}
For any $\lambda, \mu\in \check{P}$, then $\mu\prec_I \lambda$ if and only if $v_{\varpi(\mu)}\in \hat{\scr{D}}_\lambda$. 
\end{proposition}

\begin{proof}
First of all, $\mu \prec_I \lambda$ if and only if $\bar{\tau}_{-\lambda - \check{\omega}_\kappa } \prec \bar{\tau}_{-\lambda - \check{\omega}_\kappa }$. 
By Lemma \ref{stabilizer_fundamental_wt}, we have 
\[ \varpi(\lambda)=\bar{\tau}_{-\lambda - \check{\omega}_\kappa }(\check{ \Lambda}_\kappa ), \quad   \varpi(\mu)=\bar{\tau}_{-\lambda - \check{\omega}_\kappa }(\check{ \Lambda}_\kappa ).\]
Hence $\hat{\scr{D}}_\lambda= U(\Lb) \cdot v_{ \varpi(\lambda) }$. 
Lastly, in view of Proposition \ref{sect4.1_lem} we conclude that $\mu \prec_I \lambda$ if and only if $v_{\varpi(\mu)}\in \hat{\scr{D}}_\lambda  $.
\end{proof}

Let $\tt{P}( \scr{H}_\kappa )$ denote the weight system of the integrable representation $ \scr{H}_\kappa $. A weight $\hat{\mu}\in \tt{P}( \scr{H}_\kappa )$ is called maximal if $\hat{\mu}+\check{\delta} \not\in \tt{P}( \scr{H}_\kappa ) $. Let ${\rm Max}(\scr{H}_\kappa) $ denote the set of all maximal weights in $\tt{P}( \scr{H}_\kappa )$. The weight system $\tt{P}( \scr{H}_\kappa )$ can be completely described by the following lemma (cf.\,\cite[\S 12.6]{Ka}). 
\begin{lemma}
\label{Weight_System_lem}
For any $\kappa\in \hat{M}$, we have
\begin{enumerate}
\item ${\rm Max}(\scr{H}_\kappa)=  \{ \varpi(\lambda) \li    \lambda\in \check{Q}_\kappa \} $;
\item  ${\tt{P}}( \scr{H}_\kappa)=\bigcup_{\lambda\in \check{Q}_\kappa} \{ \varpi(\lambda)-n\check{\delta}  \li   n\in \mathbb{Z}^+ \}  $ is a disjoint union.
\end{enumerate}
\end{lemma}

Any weight in $\tt{P}( \scr{H}_\kappa )$ is of the form $\check{\Lambda}_0-\lambda+m \check{ \delta}$, for some integer $m$. We define a twisted version of projection 
\[   {\tt{p}}: { \tt{P}}( \scr{H}_\kappa )\to \check{P}, \text{ given by }  \check{\Lambda}_0-\lambda+m \check{ \delta} \mapsto \lambda. \]

Let ${\tt P}(\hat{\scr{D}}_\lambda )$ denote the weight system of the affine Demazure module $\hat{\scr{D}}_\lambda $. 
\begin{theorem}
\label{duality_thm}
The map $\tt{p}$ maps  $ {\tt P}(\hat{\scr{D}}_\lambda )  )$ onto  $\Psi(\lambda)$, and $\tt{p}$ admits a canonical section $\lambda\mapsto  \varpi(\lambda)$.

\end{theorem}
\begin{proof}

Set 
\[ \frak{H}:= \oplus_{n>0} ( \Lg )_{\pm  n \check{ \delta} }   \oplus   {^L}\frak{h}  \oplus   \bb{C} \check{c}, \text{ and } \frak{H}^\pm:= \oplus_{n>0} ( \Lg )_{\pm n \delta}.\]
Then  $\frak{H}$ is an Heisenberg algebra with the center $\check{c}$. By the Frenkel-Kac construction (for untwisted affine types, see \cite{FK}; for $A^{(2)}_{2n-1}, D^{(2)}_{n+1}, E^{ (2)}_{6}$, see \cite[Theorem I.2.25]{Fr} \cite[\S 7]{FLM}; for $D^{(3)}_4$, see \cite[\S B.8 ]{BT}), the representation $\scr{H}_\kappa$ can be realized as 
\begin{equation}
\label{Frenkel_Kac_con}
 \scr{H}_\kappa\simeq  S(\frak{H}^- )\otimes \bb{C}[\check{Q}_\kappa   ], \end{equation}
where $S(\frak{H}^- )$ is the symmetric algebra of $\frak{H}^-$, and $\bb{C}[\check{Q}_\kappa    ]$ consists of linear combinations of $e^{\lambda}, \lambda\in  \check{Q}_\kappa    $. Moreover 
\[ \scr{H}_\kappa(  \varpi(\lambda) - n\check{ \delta} )\simeq  S(\frak{H}^- )_{-n\check{ \delta} }\otimes e^{\lambda}, \]
where $e^\lambda$ is of weight $\varpi(\lambda)$, and 
\[ \oplus_{n \geq 0} \scr{H}_\kappa(  \varpi(\lambda) - n\check{ \delta} )  \simeq  S(\frak{H}^- )\otimes e^{\lambda}  \]
is a free $U(\frak{H}^- )$-module of rank 1. 

By Lemma \ref{Weight_System_lem}, any weight vector in $\hat{\scr{D}}_\lambda$ is of weight $\varpi(\mu)-m\check{\delta}$ for some integer $m$ and $\mu\in \check{P}$. 
To show that  $\tt{p}$ maps  $ {\tt P}(\hat{\scr{D}}_\lambda )  $ onto  $\Psi(\lambda)$, it suffices to show that 
for any weight vector $x_{ \varpi(\mu) - m\check{ \delta} }$ in $ \hat{\scr{D}}_\lambda $ of weight $ \varpi(\mu) - m\check{ \delta}$, the maximal vector $v_{\varpi(\mu)}$ is also an element in $\hat{\scr{D}}_\lambda$. 

By the construction in (\ref{Frenkel_Kac_con}), $x_{ \omega(\mu) - m\check{ \delta} }$ can be written as $P\cdot v_{\varpi(\mu)}$, where $P$ is an element of weight $-m\check{\delta}$ in the enveloping algebra  $U(\frak{H}^-)$ of $\frak{H}^-$. Let $(\cdot,\cdot )$ be the contravariant form on $\scr{H}_\kappa$, and let $\sigma$ be the Cartan involution on $\Lg$. Then 
\[ 0\not= (x_{ \varpi(\mu) - m\check{ \delta} },x_{ \varpi(\mu) - m\check{ \delta} } ) = (P\cdot v_{\varpi(\mu)},x_{ \varpi(\mu) - m\check{ \delta} } ) =    (v_{\varpi(\mu)}, \sigma(P)x_{ \varpi(\mu) - m\check{ \delta} }).\]
By the one-dimensionality of the weight space $ \scr{H}_\kappa(\varpi(\mu)) $, $ \sigma(P)x_{ \varpi(\mu) - m\check{ \delta} } = c  v_{\varpi(\mu)} $ for some nonzero constant $c\in \bb{C}$, since $\sigma(P)\in U(\frak{H}^+ )_{m\check{\delta}}$. Here we use the fact that $\sigma$ maps $\frak{H}^-$ to $\frak{H}^+$. When $\Lg$ is of unitwisted affine types, this fact follows from \cite[\S7.6]{Ka}. This fact holds for twisted affine types as well, see \cite[\S8.3]{Ka}.

 Since $\frak{H}^+ \subset  \Lb$, it follows that $v_{\varpi(\mu)}\in   \hat{\scr{D}}_\lambda $. Then by Proposition \ref{Bruhat_Dem_affine_prop}, we have $\mu\prec_I \lambda$. In other words, $\mu\in \Psi(\lambda)$.
 This concludes the proof. 
\end{proof}
\begin{remark}
In the proof of Theorem \ref{duality_thm}, we crucially use the Frenkel-Kac construction of basic representations for $\Lg$. Note that this construction only works for affine Kac-Moody algebras of type $X_n^{(r)}$, where $r=1,2,3$ and $X_n$ is of type $A,D,E$. From Table (\ref{Table_Dual_diagram_affine}), we see that $\Lg$ exhausts all cases except $A_{2n}^{(2)}$. 
\end{remark}
\begin{example}
By Lemma \ref{Weight_System_lem}, it is clear that ${\tt p}:  {\tt{P}}( \scr{H}_\kappa )\to \check{P}$ is not one-to-one. In fact ${\tt{p}}:  {\tt P}(\hat{\scr{D}}_\lambda )  )\to \Psi(\lambda)$ is not one-to-one as well. For example when $\fg=sl_2$ and $\kappa=0$, consider the affine Demazure module 
$ \hat{\scr{D}}_{-2\check{\alpha } } $, where $\alpha$ is the simple root of $\fg$. One can check that $0\not=(e_{\check{\alpha}})^2 v_{\varpi(-2\check{\alpha}) }$ has weight $\check{\Lambda}_0-2\delta$. The following example describes all weights appearing in $ \hat{\scr{D}}_{-2\check{\alpha } } $.
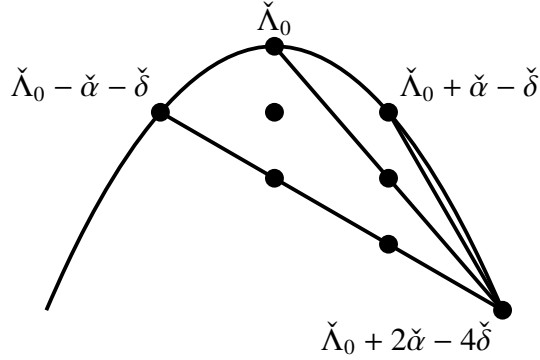
\begin{center}
	\begin{tikzpicture}[scale=0.50]
	\draw[black, line width = 0.50mm]   (-6,-7) parabola bend (0,0) (6,-7);
	\draw[black, line width = 0.50mm] (6,-7) -- (0,0);
	\draw[black, line width = 0.50mm] (6,-7) -- (-3, -1.75);
	\draw[black, line width = 0.50mm] (6,-7) -- (3, -1.75);
	\node at (6,-7)[circle,fill,inner sep=2.5pt]{};
	\node at (3,-1.75)[circle,fill,inner sep=2.5pt]{};
	\node at (0,0)[circle,fill,inner sep=2.5pt]{};
	\node at (-3, -1.75)[circle,fill,inner sep=2.5pt]{};
	\node at (0,-1.75)[circle,fill,inner sep=2.5pt]{};
	\node at (0,-3.5)[circle,fill,inner sep=2.5pt]{};
	\node at (3,-3.5)[circle,fill,inner sep=2.5pt]{};
	\node at (3,-5.25)[circle,fill,inner sep=2.5pt]{}; 
	\draw (6,-7) node[anchor=north east] {$\check{\Lambda}_0+2\check{\alpha}-4\check{\delta}$};
	\draw (0,0) node[anchor=south] {$\check{\Lambda}_0$};
	\draw (3,-1.75) node[anchor=south west] {$\check{\Lambda}_0+\check{\alpha}-\check{\delta}$};
	\draw (-3, -1.75) node[anchor=south east] {$\check{\Lambda}_0-\check{\alpha}-\check{\delta}$};
	\end{tikzpicture}
\end{center}
\end{example}

Let $e_{ \check{\alpha} + m r_\alpha \check{\delta }  }$ be a root vector in $\Lg$ corresponding to the  root $\check{\alpha} + m r_\alpha  \check{\delta} \in {^L}\hat{\Phi}^+$. Via the bijection $\eta$ given in Section \ref{sect_duality1}, by Lemma \ref{eta_equivariant} the coroot of  $\check{\alpha} + m r_\alpha  \check{\delta}$ is $\alpha+m \delta\in \hat{\Phi}^+$. 

The following lemma gives a representation theoretic interpretation for Lemma \ref{sect_lem_1}. 

\begin{lemma}
\label{sect4.2_lem2}
For any $\alpha\in \Phi^+$, 
\begin{enumerate}
\item if $\langle \lambda, \alpha  \rangle>0$, then for any $0< k \leq   \langle \lambda, \alpha  \rangle$, 
 \begin{equation}
 \label{Kashiwara_op_1}
  (e_{\check{\alpha}} ) ^k\cdot v_{\varpi(\lambda)} \not= 0, \text{ and } \, {\tt{p}} ( { {\rm wt}} ( (e_\alpha)^k\cdot v_{\varpi(\lambda)} ) ) =\lambda-k\check{\alpha} ; \end{equation}
\item  if $\langle \lambda, \alpha  \rangle<-1$, then for any $0< k <-  \langle \lambda, \alpha  \rangle$, 
\begin{equation}
\label{Kashiwara_op_2}
  (e_{-\check{\alpha}+ r_\alpha \check{\delta}} )^k\cdot v_{\varpi(\lambda)} \not= 0, \text{ and }\, {\tt{p}} ( { {\rm wt}}(e_{-\check{\alpha}+ r_\alpha \check{\delta}} )^k\cdot v_{\varpi(\lambda)} ))=\lambda+k\check{\alpha},\end{equation}
\end{enumerate}
where ${\rm  wt}(\cdot)$ denote the weight of a weight vector.
\end{lemma}
\begin{proof}
With respect to any $sl_2$-triple $\{  e_{ \check{\alpha} + k r_\alpha \check{\delta }  }, e_{ -\check{\alpha} - k r_\alpha \check{\delta }  }, \alpha+ k \delta \}$ associated to  $\check{\alpha} + k r_\alpha  \check{\delta} \in \hat{\Phi}^+$, the extremal vector $v_{\varpi(\lambda)}$ is either a highest weight vector or lowest weight vector. In part (1), by (\ref{maximal_wt_formula}) 
\[  \langle \varpi(\lambda), \alpha  \rangle= - \langle \lambda, \alpha  \rangle <0, \]
the vector $v_{\varpi(\lambda)}$ is a weight vector of lowest weight $ - \langle \lambda, \alpha  \rangle$ with respect to the $sl_2$-triple associated to $\check{\alpha}$. Hence the statement ( \ref{Kashiwara_op_1}) holds. 

In part (2), since 
\[  \langle \varpi(\lambda), -\alpha+ \delta  \rangle= 1 + \langle \lambda, \alpha  \rangle <0, \]
the vector $v_{\varpi(\lambda)}$ is again a weight vector of lowest weight $1 + \langle \lambda, \alpha  \rangle$ with respect to the $sl_2$-triple associated to the positive root $-\check{\alpha}+r_\alpha \check{\delta}$. Hence the statement ( \ref{Kashiwara_op_2}) holds as well. 

\end{proof}



\begin{remark}
For any $\lambda\in \check{Q}_\kappa$, let $\gamma$ denote the unique dominant coweight that is translated by the Weyl group $W$ from $- \lambda$. Set $V_{\gamma}: = U({^L \fg} ) \cdot v_{\varpi(\lambda) } \subset \scr{H}_\kappa$. It is well-known that $V_{\gamma}$ is 
a finite dimensional irreducible representation of highest weight $\gamma$. The vector $v_{\varpi(\lambda) }\in V_{\gamma} $ is an extremal vector of weight $-\lambda$. 
Let ${^L}\fb$ denote the Borel subalgebra of ${^L}\fg$ obtained as $ {^L}\fb:= \Lb \cap  {^L}\fg $. Then $D_{-\lambda}:= U({^L}\fb )v_{\varpi(\lambda)}$ is a Demazure module in $V_{\gamma} $ in the usual sense. 
It is proved in \cite{Sch,Bo} that the multiplicity of the weight space $D_{-\lambda}(-\mu)$ is equal to the number of top components of the intersection $N(\scr{K})\cdot L_\mu \cap  I\cdot L_\lambda$. It is a natural question to ask if there is similar phenomenon for the affine Demazure module $ \hat{\scr{D}}_\lambda$.
\end{remark}

\section{Further study of the set $\Psi(\lambda)$}
\label{sect4}
\subsection{ The partial order $\prec_I$ on a chamber of the coweight lattice}

Let $\mathfrak{C}$ be the dominant chamber of the coweight lattice $\check{P}$ determined by $T \subset B$, i.e. 
\[  \mathfrak{C}=\{ \lambda\in \check{P}  \li     \langle \lambda, \alpha \rangle\geq 0, \text{ for any } \alpha\in \Phi^+  \}.\]
Set $\frak{C}_w:=w( \frak{C} )$ for any $w\in W$. Then $\check{P}$ is the union of all chambers $\frak{C}_w$ indexed by $w\in W$.

\begin{theorem}
\label{sect4.1_Thm}
For any coweights $\lambda, \mu\in \frak{C}_w$, $\mu\prec_I \lambda$ if and only if $\lambda-\mu$ is a positive sum of coroots in $w(\check{\Phi}^+)$.
\end{theorem}
\begin{proof}
We first assume that  $\mu\prec_I \lambda$, in other words  $L_\mu \subset \overline{\X}_\lambda$. Set $\lambda^+:=w^{-1}(\lambda)$ and $\mu^+:= w^{-1}(\mu)$. Then $\lambda^+,\mu^+\in \frak{C}$. 
It follows that $L_{\mu^+}\subset \overline{\Gr}_{\lambda^+}$, since $\Gr_{\lambda^+}=G(\scr{O})\cdot L_\lambda$. Hence $\lambda^+-\mu^+$ is a positive sum of coroots in $\check{\Phi}^+$. Equivalently, $\lambda-\mu$ is a positive sum of coroots in $w(\check{\Phi}^+)$.

Now we prove the converse. Assume that $\lambda-\mu$ is a positive sum of coroots in $w(\check{\Phi}^+)$. Equivalently, $\lambda^+-\mu^+$ is a positive sum of coroots in $\check{\Phi}^+$. By a result due to Stembridge and Steinberg (cf.\,\cite[Corollary 2.7]{Stg}), there exists a sequence of dominant coweights
\[ \mu^+=\lambda^+_0, \lambda^+_1, \cdots, \lambda^+_{k-1}, \lambda^+_k=\lambda^+, \]
such that for each $1\leq i\leq k$, $\lambda_i^+= \lambda^+_{i-1}+\check{\beta}_i$ where $\check{\beta}_i\in \check{ \Phi}^+$. Set $\lambda_i:=w(\lambda^+_i)$ for each $i$. Then $\lambda_i=\lambda_{i-1}+ w(\check{\beta}_i)$. Since 
\[ \langle \lambda_i, w(\beta_i )   \rangle= \langle \lambda_{i-1}+ w(\check{\beta}_i), w(\beta_i )   \rangle=        \langle \lambda^+_{i-1}+ \check{\beta}_i, \beta_i    \rangle=   \langle \lambda^+_{i-1}, \beta_i \rangle +2 >0, \]
by Lemma \ref{sect_lem_1}, $\lambda_{i-1}= \lambda_i- w(\check{\beta}_i) \in \Psi( \lambda_i  )$. Hence $\lambda_{i-1}\prec_I  \lambda_I$. Inductively, we have $\mu\prec_I \lambda$.
  
\end{proof}
This theorem completely describes the partial order $\prec_I$ on any fixed chamber. Moreover 
from this theorem, we immediately get the following corollary.
\begin{corollary}
\label{sect4.1_cor}
For any two coweights $\lambda, \mu$ in the same chamber, $\mu\prec_I \lambda$ if and only if $w(\mu)\prec_I  w(\lambda)$.
\end{corollary}

\subsection{R-operators and cover relations}
\label{sect_4.2}

For each $\alpha\in \Phi^+$, we define the operator $R_\alpha$ on $\check{P}$, 
\begin{equation}
R_\alpha(\lambda)=\begin{cases}   s_\alpha(\lambda):=\lambda- \langle \lambda, \alpha  \rangle\check{\alpha}, \quad  \text{ if } \langle \lambda, \alpha  \rangle\geq 0   \\
s_\alpha(\lambda)-\check{\alpha}=:\lambda- \langle \lambda, \alpha  \rangle\check{\alpha}-\alpha, \quad  \text{ if } \langle \lambda, \alpha  \rangle< 0 
\end{cases}.
\end{equation}
By Lemma \ref{sect_lem_1}, for any $\lambda\in \check{P}$, we always have $R_\alpha(\lambda)\in \Psi(\lambda)$. 

The following theorem shows that the set $\Psi(\lambda)$ can be obtained by repeating applying $R$-operators starting at  $\lambda$. 
\begin{theorem}
\label{R_op_Thm}
For any coweight $\lambda\in \check{P}$, we have 
\[ \Psi(\lambda)= \{ R_{\beta_1}R_{\beta_2}\cdots  R_{\beta_k}(\lambda)  \li     k\in \bb{N}, \, \beta_1, \beta_2,\cdots, \beta_k\in \Phi^+    \}. \]
\end{theorem}
\begin{proof}
By Lemma \ref{sect_lem_1}, for any $\lambda\in \check{P}$, we always have $R_\alpha(\lambda)\prec_I \lambda$. Therefore for any sequence of positive roots $\beta_1, \cdots, \beta_k$, we have $R_{\beta_1}R_{\beta_2}\cdots  R_{\beta_k}(\lambda)\in \Psi(\lambda) $. 

To show the other inclusion, we first prove the following general fact:  for any $\mu,\mu'\in \check{P}$, if $\mu'\in S(\mu,\alpha)$ for some positive root $\alpha\in \Phi^+$ (recall $S(\mu,\alpha)$ defined in (\ref{sect2.1_set_S})), then $\mu'$ can be written as $(R_{
\alpha})^k (\mu)$ for some integer $k$. We first assume that $  \langle \mu, \alpha   \rangle \geq 0  $. One can check that for any integer $m\geq 0$,
\begin{equation}
\label{lem_case_1}
\mu-m\check{\alpha}= \begin{cases} (R_\alpha)^{2m}(\mu), \quad   \text{ if }\, 0\leq  m\leq  \frac{ 1}{2}    \langle \mu, \alpha   \rangle \\
(R_\alpha)^{2( \langle \lambda, \alpha   \rangle-    m)+1}(\mu), \quad   \text{ if }\, \frac{    1}{2}\langle \mu, \alpha   \rangle < m \leq   \langle \lambda, \alpha   \rangle 
     \end{cases}. \end{equation}
Now we assume that $ \langle \mu, \alpha   \rangle < 0  $, one can check similarly 
\begin{equation}
\label{lem_case_2}
\mu+ m\check{\alpha}= \begin{cases} (R_\alpha)^{2m}(\mu), \quad   \text{ if } \, 0\leq m<-  \frac{ 1}{2} \langle \mu, \alpha   \rangle \\
(R_\alpha)^{2( - \langle \mu, \alpha   \rangle-    m)-1}(\mu), \quad   \text{ if }\, - \frac{ 1}{2} \langle \mu, \alpha   \rangle \leq m  \leq   - \langle \mu, \alpha   \rangle-1        
     \end{cases}. \end{equation}
 From the computations (\ref{lem_case_1}) (\ref{lem_case_2}), we see that this theorem follows from Theorem \ref{Thm_1}.
\end{proof}

For $\alpha\in \Phi^+$ and $k\in \bb{Z}$, let $H_{\alpha,k}$ denote the hyperplane $\{  \lambda\in \check{P}  \li     \langle  \lambda, \alpha   \rangle=k   \}$.
\begin{lemma}
\label{sect4.2_lem1}
Assume that $\mu, \lambda $  are in the same chamber  $\mathfrak{C}_w$ and $\mu, \lambda\not \in H_{\alpha,-1}$ for some positive root $\alpha\in \Phi^+$, if $\mu \prec_I  \lambda$, then $R_{\alpha}(\mu) \prec_I  R_\alpha(\lambda)$.
\end{lemma}

\begin{proof}
Since $\alpha$ is either an element in $w(\Phi^+)$ or an element in $w(\Phi^-)$, it follows that either $\langle \lambda, \alpha  \rangle \geq 0$ and $\langle \mu, \alpha  \rangle \geq 0$, or $\langle \lambda, \alpha  \rangle \leq  0$ and $\langle \mu, \alpha  \rangle \leq  0$. 
By the assumption $\mu,\lambda\not \in H_{\alpha,-1}$, it follows that either $\langle \lambda, \alpha  \rangle \geq 0$ and $\langle \mu, \alpha  \rangle \geq 0$, or  $\langle \lambda, \alpha  \rangle <  0$ and $\langle \mu, \alpha  \rangle <  0$. 

 In the first case $R_\alpha(\mu)\prec_I  R_\alpha(\lambda)$, which follows from Corollary \ref{sect4.1_cor}. In the second case, 
 \[ R_\alpha(\lambda)=s_\alpha(\lambda)-\check{\alpha}, \text{ and }  R_\alpha(\mu)=s_\alpha(\mu)-\check{\alpha}. \]
Hence by Theorem \ref{sect4.1_Thm}, $R_\alpha(\lambda)-R_\alpha(\mu)$ is a positive sum of coroots in $s_\alpha w(\check{\Phi}^+)$. Then by Theorem \ref{sect4.1_Thm} again, $R_\alpha(\mu) \prec_I  R_\alpha(\lambda$). This concludes the proof.
\end{proof}

In the following, we introduce the notion of $\alpha$-regularity $n$-regularity for a coweight. This notion will allow us to cross the wall $H_\alpha:=H_{\alpha,0}$. 
\begin{definition}
\label{regularity_def}
\begin{enumerate}
\item We say that a coweight $\lambda$ is $\alpha$-regular, if $ \langle \lambda, \alpha  \rangle> 0 $, or $ \langle \lambda, \alpha  \rangle< 0 $ and $\lambda+\check{\alpha}\in \frak{C}_{w^\lambda}$, where $w^\lambda$ is the minimal element in $W$ such that $\lambda\in \frak{C}_{w^\lambda}$.  
\item We say that  $\lambda$ is $n$-regular for a nonnegative integer $n$, if $\langle \lambda,  w^\lambda(\alpha_i)  \rangle \geq n$ for any simple root $\alpha_i$.
\end{enumerate}
\end{definition}
Set \[
r=  \begin{cases}     1, \quad  \text{ if  $\Gamma$ is simply-laced};  \\
   2, \quad  \text{ if  $\Gamma$ is non simply-laced but } \Gamma\not= G_2 ;  \\
3, \quad  \text{ if  $\Gamma=G_2$ }
\end{cases}.
\]

\begin{lemma}
 If $\lambda$ is $r$-regular, then $\lambda$ is $\alpha$-regular for any positive root $\alpha$.
\end{lemma}
\begin{proof}
 By the assumption of $r$-regularity, it is easy to see that $\langle \lambda, \alpha \rangle\not= 0$ for any root $\alpha$, since $\alpha$ is always a positive (or a negartive) summation of $w^\lambda(\alpha_i)$. 
We may only consider the case when $\langle \lambda, \alpha \rangle <0$. 
Note that we always have $|\langle  \check{\alpha}, w^\lambda(\alpha_i)   \rangle| \leq r$ (cf.\,\cite[\S 9.4]{Hu1}).  It follows that 
\[ \langle \lambda+\check{\alpha}, w^\lambda(\alpha_i)   \rangle \geq  0, \text{ for each simple root } \alpha_i. \] 
Equivalently, $\lambda+\check{\alpha}\in \frak{C}_{w^\lambda}$. It finishes the proof of the lemma.
\end{proof}

\begin{lemma}
\label{sect4.2_lem2}
Let $\alpha $ be a positive root in $\Phi^+$. Let $\lambda$ be an $\alpha$-regular coweight. Then $\lambda\not\in H_{\alpha,-1}$ and
$R_\alpha(\lambda)\in \frak{C}_{s_\alpha w}$, where $w$ is the minimal element in $W$ such that $\lambda\in \frak{C}_w$.
\end{lemma}
 \begin{proof}
 First of all we show that $\lambda\not \in H_{\alpha,-1}$. 
 Assume that $\langle \lambda, \alpha  \rangle\leq -1$. 
 Then  $w^{-1}(\alpha)\in \Phi^-$, since $\lambda\in \frak{C}_w$. Write $-\alpha=\sum c_i w (\alpha_i)$ as a positive linear combination 
  of the $w$-translations $w(\alpha_i)$ of simple roots. Then
 \[  \langle \lambda+ \check{\alpha}, \alpha   \rangle =- \sum c_i  \langle \lambda+ \check{\alpha}, w( \alpha_i)   \rangle \leq  0. \]
  By the $\alpha$-regularity of $\lambda$, it follows that $\langle \lambda, \alpha   \rangle \leq  -2 $. It follows that $\lambda\not \in H_{\alpha,-1}$. 
  
  We now show that  $R_\alpha(\lambda)\in \frak{C}_{s_\alpha w}$.
 When $ \langle \lambda, \alpha  \rangle> 0 $,  it is obvious. 
 Set $\beta:= - w^{-1}(\alpha)\in \Phi^+$. Then $\check{\beta}=-w^{-1}(\check{\alpha})$. For any simple root $\alpha_i$, 
 \[\langle \check{\beta}, \alpha_i  \rangle\leq 2\]
To show that $R_\alpha(\lambda)\in \frak{C}_{s_\alpha w}$, it suffices to check that, for any simple root $\alpha_i$, $\langle  R_\alpha(\lambda), s_\alpha w( \alpha_i ) \rangle\geq 0$, which follows since
 \begin{align*} 
 \langle  R_\alpha(\lambda), s_\alpha w( \alpha_i ) \rangle  &= \langle  s_\alpha(\lambda)-\check{\alpha}, s_\alpha w( \alpha_i ) \rangle        \\
                                                     &= \langle  \lambda+\check{\alpha}, w( \alpha_i ) \rangle    
                                                      \geq 0.
 \end{align*}
\end{proof}

 The following result immediately follows from Lemma \ref{sect4.2_lem1} and Lemma \ref{sect4.2_lem2}.
 \begin{proposition}
 \label{stability_prop}
 Let $\alpha$ be a positive root. For any $\alpha$-regular coweights $\mu,\lambda$ in $\frak{C}_w$, if $\mu\prec_I \lambda$, then $R_\alpha(\mu), R_\alpha(\lambda)\in \frak{C}_{s_\alpha w}$ and $R_\alpha(\mu) \prec_I R_\alpha(\lambda)$.
 
 \end{proposition}
 

Following \cite{BGG}, we write $y \xrightarrow{\alpha} w$ for any $y,w\in W$,  if $y=s_{\alpha}w$ and $\ell(w)=\ell(y)+1$ for some reflection $s_\alpha$. Similarly, we write $\mu\xrightarrow{\alpha} \lambda$ for any $\mu,\lambda\in \check{P}$, if $\mu=R_\alpha(\lambda)$ and $\dim \X_\lambda=\dim \X_\mu+1$ for some positive root $\alpha$. As shown in Theorem \ref{R_op_Thm}, the $R$-operators generate the closure relations for $\prec_I$. So we see that the covering relations of $\prec_I$ on $\check{P}$ must all be of the form $\mu\xrightarrow{\alpha}\lambda$. 

Recall that $\dim \X_\lambda=\ell(\tau_{-\lambda} w^\lambda )$ where $\ell$ is the length function defined in (\ref{length_formula}), and $w^\lambda$ is the minimal element in $W$ such that $\lambda\in    \frak{C}_{w^\lambda}$.
Set $\lambda^+:=(w^\lambda)^{-1}(\lambda)$; by definition $\lambda^+$ is the dominant translate of $\lambda$. The following formula is a consequence of  (\ref{length_formula}),
\begin{equation}
\label{dim_formmula}
 \dim \X_\lambda =2\langle \lambda^+ ,\rho \rangle - \ell(w^\lambda). \end{equation}
From this formula, we immediately see that $\mu\xrightarrow{\alpha} \lambda$ if and only if $\mu=R_\alpha(\lambda)$ and 
\begin{equation}
\label{cover_formula}
\ell(w^\lambda)-\ell(w^\mu)= 2\langle \lambda^+ - \mu^+, \rho \rangle-1 .\end{equation}
This characterization is a bit complicated, as we need to determine $\lambda^+, \mu^+$ and the lengths of $w^\lambda$ and $w^\mu$.  In the following two propositions, we give a simpler characterization of the cover relation $R_\alpha(\lambda)\xrightarrow{\alpha} \lambda$ when $\lambda$ is $\alpha$-regular.
\begin{proposition}
\label{Cover_prop_1}
Let $\mu,\lambda$ be two coweights in  $\check{P}$ such that $\mu=R_\alpha(\lambda)$ for some positive root $\alpha$. Assume that $\langle  \lambda, \alpha  \rangle>0$. Then   $\mu\xrightarrow{\alpha} \lambda$ if and only if $w^{\lambda} \xrightarrow{\alpha} s_{\alpha} w^{\lambda}$.
\end{proposition}
\begin{proof}
Observe that $w^\lambda \prec s_\alpha w^\lambda$ and $(w^\lambda)^{-1}(\alpha)\in \Phi^+$, which follow from the assumption that $\langle  \lambda, \alpha  \rangle>0$. We also notice that $\lambda^+=\mu^+$ in this case.

We first prove the direction $``\impliedby"$. It is enough to show that $w^\mu=s_\alpha w^\lambda$. By the formula $(\ref{length_formula})$, for any $w$ and $\lambda\in \frak{C}_w$, we always have 
\[ \ell(\tau_{-\lambda} w)\geq    2\langle \lambda^+, \rho \rangle-\ell(w) .\]
In fact, the equality holds if and only if $w=w^\lambda$. We now show that $s_{\mu}=s_{\alpha}w^{\lambda}$ by by contradiction. Assume that $s_\alpha w^\lambda\not= w^\mu$, implying the following inequality:
\begin{equation}
\label{inequality_cover}
 \ell(\tau_{-s_\alpha(\lambda)}s_\alpha w^\lambda)> 2\langle \lambda^+, \rho \rangle-\ell(s_\alpha w^\lambda)= 2\langle \lambda^+, \rho \rangle-\ell( w^\lambda)-1.\end{equation}
On the other hand, we view $\tau_{-\lambda} w^\lambda$ as an element in the extended affine Weyl group $\tilde{W}_{\rm aff}$.  By the formula (\ref{formula_affine_2}),  
\[ (\tau_{-\lambda}w^\lambda)^{-1}(\alpha)=(w^\lambda)^{-1}(\alpha)-\langle \lambda, \alpha \rangle \check{ \delta} 
 \]
 is a negative affine root.  It follows that $s_\alpha \tau_{-\lambda}w^\lambda \prec \tau_{-\lambda}w^\lambda $. In particular, 
\[ \ell(s_\alpha \tau_{-\lambda}w^\lambda)< \ell( \tau_{-\lambda}w^\lambda)= 2\langle \lambda^+, \rho \rangle-\ell(w^\lambda) .\]
It contradicts with the inequality (\ref{inequality_cover}), since $s_\alpha \tau_{-\lambda}w^\lambda=\tau_{-s_\alpha(\lambda)}s_\alpha w^\lambda$.

The direction $``\implies"$ is obvious.

\end{proof}

\begin{proposition}
\label{Cover_prop_2}
For any two coweights $\mu,\lambda$ in  $\check{P}$ with $\mu=R_\alpha(\lambda)$ for some positive root $\alpha$, assume that $\langle \lambda,\alpha \rangle<0$ and $\lambda$ is $\alpha$-regular. Then $\mu\xrightarrow{\alpha} \lambda$ if and only $\ell(w^{\lambda})-\ell(w^{\mu})=-2\langle \check{\alpha}, w_\lambda(\rho)   \rangle-1$. 
\end{proposition}
\begin{proof}
By Lemma \ref{sect4.2_lem2}, $\mu=R_\alpha(\lambda)\in \frak{C}_{s_\alpha w^\lambda}$. It follows that $\mu^+=(w^\lambda)^{-1}(\lambda+\check{\alpha})$. Then 
 \[\langle \lambda^+-\mu^+, 2\rho  \rangle= -2\langle (w^\lambda)^{-1}(\check{\alpha}), \rho  \rangle=-2\langle \check{\alpha}, w^\lambda(\rho)   \rangle.   \]
Hence $\mu\xrightarrow{\alpha}\lambda$ if and only if $\ell(w^{\lambda})-\ell(w^{\mu})=-2\langle \check{\alpha}, w_\lambda(\rho)   \rangle-1$.

\end{proof}

Note that if $\lambda$ is $(r+1)$-regular, then $w^\mu=s_\alpha w^\lambda$; by the same argument as in Lemma \ref{sect4.2_lem2}, it is easy to check that if $\lambda$ is $(r+1)$-regular, then $\mu=R_\alpha(\lambda)\in \frak{C}_{s_\alpha w^\lambda}$ is $1$-regular. Hence $w^\mu=s_\alpha w^\lambda$. This fact, along with Proposition \ref{Cover_prop_2}, imply the following corollary.
\begin{corollary}
With the same setup as in Proposition \ref{Cover_prop_2}, assume that $\langle \lambda, \alpha \rangle<0$ and  $\lambda$ is $(r+1)$-regular. If $R_\alpha(\lambda)\xrightarrow{\alpha}\lambda$, then $s_\alpha w^\lambda \xrightarrow{\alpha} w^\lambda$ if and only if $-(w^\lambda)^{-1}(\alpha)$ is a simple root.
\end{corollary}

Now for any $\lambda\in \check{P}$, let $\Psi(\lambda)_\partial$ denote the subset of $\Psi(\lambda)$ consisting of $\mu$ such that $\mu\xrightarrow{\alpha} \lambda$ for some $\alpha$. Geometrically, if $\mu \in \Psi(\lambda)_{\partial}$ then $\overline{\X}_{\mu}$ is an irreducible divisor in $\overline{\X}_{\lambda}$. In general, the set $\Psi(\lambda)_\partial$ consist of two types of elements $R_\alpha(\lambda)$, those for which $\langle \lambda, \alpha \rangle>0$ and those for which $\langle \lambda, \alpha \rangle <0$.
\begin{remark}
Prop.\ref{Cover_prop_1} and Prop.\ref{Cover_prop_2} are special cases of \cite[Proposition 4.1]{LS} when $\lambda$ is super-regular in the sense of \cite{LS}. 
\end{remark}

\begin{example}
Let $G$ be of type $A_2$. Let $\alpha_1$ and $\alpha_2$ be the two simple roots. 
\begin{enumerate}
\item  Take $\lambda=\check{\alpha}_1+\check{\alpha}_2$. Then  $\Psi(\lambda)_\partial=\{\check{\alpha}_1,\check{\alpha}_2\}$.
\item Take $\lambda=-(\check{\alpha}_1+\check{\alpha}_2)$. Then $\Psi(\lambda)_\partial=\{ 0 \}$.
\item Take $\lambda=-2(\check{\alpha}_1+\check{\alpha}_2)$.  Then $\Psi(\lambda)_\partial=\{ -2\check{\alpha}_1-\check{\alpha}_2, -\check{\alpha}_1-2\check{\alpha}_2, \check{\alpha}_1+\check{\alpha}_2 \}$.
\item Take $\lambda=-3\check{\alpha}_1$. Then $\Psi(\lambda)_\partial=\{ 2\check{\alpha}_2-\check{\alpha}_1, -3(\check{\alpha}_1+\check{\alpha}_2 ) \}$. The dimension of $\X_\lambda$ is 10.
The following picture illustrates for example (4) the dimension of $\X_\mu$ for all possible $\mu=R_\alpha(\lambda)$:
 \begin{center}
	\begin{tikzpicture}[scale=1.00]
	\coordinate (0;0) at (0,0); 
\foreach \c in {1,...,3}{%
\foreach \i in {0,...,5}{%
\pgfmathtruncatemacro\j{\c*\i}
\coordinate (\c;\j) at (60*\i:\c);  
} }
\foreach \i in {0,2,...,10}{%
\pgfmathtruncatemacro\j{mod(\i+2,12)} 
\pgfmathtruncatemacro\k{\i+1}
\coordinate (2;\k) at ($(2;\i)!.5!(2;\j)$) ;}

\foreach \i in {0,3,...,15}{%
\pgfmathtruncatemacro\j{mod(\i+3,18)} 
\pgfmathtruncatemacro\k{\i+1} 
\pgfmathtruncatemacro\l{\i+2}
\coordinate (3;\k) at ($(3;\i)!1/3!(3;\j)$)  ;
\coordinate (3;\l) at ($(3;\i)!2/3!(3;\j)$)  ;
 }

 \foreach \i in {0,...,6}{%
\pgfmathtruncatemacro\k{\i}
\pgfmathtruncatemacro\l{15-\i}
 \pgfmathtruncatemacro\k{9-\i} 
 \pgfmathtruncatemacro\l{mod(12+\i,18)}   
 \pgfmathtruncatemacro\k{12-\i} 
 \pgfmathtruncatemacro\l{mod(15+\i,18)}   
}    
\fill [gray] (0;0) circle (2pt);
 \foreach \c in {1,...,3}{%
 \pgfmathtruncatemacro\k{\c*6-1}    
 \foreach \i in {0,...,\k}{%
   \fill [gray] (\c;\i) circle (2pt);}}  
\draw[->, black, thick,shorten >=4pt, shorten <=2pt](3;9)--(2;0);
\draw[->,black,thick,shorten >=4pt, shorten <=2pt](3;9)--(3;7);
\draw[->,black,thick,shorten >=4pt, shorten <=2pt](3;9)--(3;12);
\draw (3;9) node[anchor=south east] {$-3\check{\alpha}_1$};
\draw (3;9) node[anchor=north east] {dim 10};
\draw (3;7) node[anchor=south east] {$2 \check{\alpha}_2-\check{\alpha}_1$};
\draw (3;7) node[anchor=south west] {dim 9};
\draw (2;0) node[anchor=south west] {$2 \check{\alpha}_1$};
\draw (2;0) node[anchor=north west] {dim 7};
\draw (3;12) node[anchor=north east] {$-3(\check{\alpha}_1+\check{\alpha}_2)$};
\draw (3;12) node[anchor=north west] {dim 9};
\draw (0;0) node[anchor=north west] {$0$};
	\end{tikzpicture}
\end{center}

\end{enumerate}
\end{example}
 \subsection{Braid relations on $R$-operators }
\label{braid_sect}
\begin{proposition} 
\label{braid_prop}
Let $\lambda$ be a coweight in $\check{P}$, and let $\alpha,\beta$ be two positive roots in $\Phi$.
\begin{enumerate}
\item If $\langle \alpha, \check{\beta} \rangle =0$, then $R_{\alpha}R_{\beta}(\lambda)=R_{\beta}R_{\alpha}(\lambda)$.
\item If $\langle \alpha, \check{\beta} \rangle =-1, \langle \beta, \check{\alpha} \rangle =-1$, and $\lambda$ is away from the hyperplane $H_{ \alpha+\beta, -1}$, 
then  $ R_{\alpha}R_{\beta}R_{\alpha}(\lambda)=R_{\beta}R_{\alpha}R_{\beta}(\lambda)$. 
\item If $\langle \alpha, \check{\beta} \rangle = -1, \langle \beta, \check{\alpha} \rangle=-2$, and $\lambda$ is away from the hyperplanes $H_{\alpha+\beta,-1}$, $H_{2\alpha+\beta,-1}$ and $H_{2\alpha+\beta,-2}$, 
then $R_{\alpha}R_{\beta}R_{\alpha}R_{\beta}(\lambda)=R_{\beta}R_{\alpha}R_{\beta}R_{\alpha}(\lambda)$.
\end{enumerate}
\end{proposition}
\begin{proof}
We first show part (1). Assume that $\langle \alpha, \check{\beta} \rangle=0$. We compare $R_{\alpha}R_{\beta}(\lambda)$ with $R_{\beta}R_{\alpha}(\lambda)$. There are four cases. If $\langle \lambda, \alpha \rangle \geq 0$ and $\langle \lambda, \beta \rangle \geq 0$, then the $R$ operators coincide with the reflections $s_{\alpha}, s_{\beta}$ which commute, so we obviously have equality. Assume $\langle \lambda, \alpha \rangle <0$, $\langle \lambda, \beta \rangle \geq 0$. Now we compute 
$$R_{\alpha}R_{\beta}(\lambda)=R_{\alpha}s_{\beta}(\lambda)=s_{\alpha}s_{\beta}(\lambda)-\check{\alpha}.$$
  The other expression evaluates to $s_{\beta}s_{\alpha}(\lambda)-s_{\beta}(\check{\alpha})$ and the two expressions are equal since $s_{\beta}(\check{\alpha})=\check{\alpha}$. The case where $\langle \lambda, \alpha \rangle \geq 0$ but $\langle \lambda, \beta \rangle <0$ is identical. Lastly if $\langle \lambda, \alpha \rangle <0$ and $\langle \lambda, \beta \rangle <0$ we get expressions $s_{\alpha}s_{\beta}(\lambda)-\check{\alpha}-\check{\beta}$ and $s_{\beta}s_{\alpha}(\lambda)-\check{\beta}-\check{\alpha}$, which again are equal. 

The proof of part (2) and part (3) essentially follows by brute force computation. We include tables with all of the relevant information, as well as some sample computations.

In the tables below we make the following convention: if a square is blank, then the braid relation is satisfied. Only when the braid relation is not satisfied do we put anything in the second column. This is simply for readability.
\begin{center}

Type $A_2$: $\langle \alpha, \check{\beta} \rangle =-1, \langle \beta, \check{\alpha} \rangle =-1$  \\
	\begin{tabular}{| p{60mm} | l | l | l | l | l |}
	\hline
	Conditions & $R_{\alpha}R_{\beta}R_{\alpha}$ & $R_{\beta}R_{\alpha}R_{\beta}$ \\ \hline
	$\langle \lambda, \alpha \rangle \leq -1,$ \, $\langle \lambda, \beta \rangle \leq -1$ & $w(\lambda)-2(\check{\alpha}+\check{\beta})$ &  \\ \hline
	$\langle \lambda, \alpha \rangle = -1$, \, $\langle \lambda, \beta \rangle =0 $ & $w(\lambda)-\check{\alpha}-\check{\beta}$ & \\ \hline
	$\langle \lambda, \alpha \rangle =0$ \, $\langle \lambda, \beta \rangle = -1$ & $w(\lambda)-\check{\alpha}-\check{\beta}$ & \\ \hline
	$\langle \lambda, \alpha \rangle \geq 0$, \, $\langle \lambda, \alpha+\beta \rangle < -1 $ & $w(\lambda)-2\check{\alpha}-\check{\beta}$ &  \\ \hline
	$\langle \lambda, \alpha \rangle \geq 1$, \, $\langle \lambda, \alpha+\beta \rangle = -1$ & $w(\lambda)-2\check{\alpha}-\check{\beta}$ & $w(\lambda)-\check{\alpha}$ \\ \hline
	$\langle \lambda, \alpha+\beta \rangle \geq 0$, \, $\langle \lambda, \beta \rangle \leq -1$ & $w(\lambda)-\check{\alpha}$ &  \\ \hline
	$\langle \lambda, \alpha \rangle \geq 0$, \, $\langle \lambda, \beta \rangle \geq 0$ & $w(\lambda)$ &  \\ \hline
	$\langle \lambda, \beta \rangle \geq 0$, \, $\langle \lambda, \alpha+\beta \rangle <-1$ & $w(\lambda)-2\check{\beta}-\check{\alpha}$ & \\ \hline
	$\langle \lambda, \alpha+ \beta \rangle=-1$, \, $\langle \lambda, \beta \rangle \geq 1$ & $w(\lambda)-\check{\beta}$ & $w(\lambda)-2\check{\beta}-\check{\alpha}$ \\ \hline
	$\langle \lambda, \alpha+\beta \rangle \geq 0$, \, $\langle \lambda, \alpha \rangle \leq -1$ & $w(\lambda) -\check{\beta}$ & \\ \hline
	\end{tabular}
\end{center}
\begin{center}
Type $B_2$: $\langle \alpha, \check{\beta} \rangle = -1, \langle \beta, \check{\alpha} \rangle=-2$\\
	\begin{tabular}{| p{60mm} | l | l | l | l | l |}
	\hline
	Conditions & $R_{\alpha}R_{\beta}R_{\alpha}R_{\beta}$ & $R_{\beta}R_{\alpha}R_{\beta}R_{\alpha}$  \\ \hline
	$\langle \lambda, \alpha \rangle \leq -1,$ \, $\langle \lambda, \beta \rangle <0$ & $w(\lambda)-3\check{\alpha}-4\check{\beta}$ &  \\ \hline
	$\langle \lambda, \alpha \rangle = 0$, \, $\langle \lambda, \beta \rangle = -2 $ & $w(\lambda)-2\check{\alpha}-3\check{\beta}$ &  \\ \hline
	$\langle \lambda, \beta \rangle =0$, \, $\langle \lambda, \alpha \rangle = -1$ & $w(\lambda)-2\check{\alpha}-2\check{\beta}$ & \\ \hline
	$\langle \lambda, \beta \rangle = -1$, \, $\langle \lambda, \alpha \rangle =0$ & $w(\lambda)-\check{\alpha}-2\check{\beta}$ &  \\ \hline
	$\langle \lambda, \alpha+\beta \rangle<-1$, \, $\langle \lambda, \beta \rangle \geq 0$ & $w(\lambda)-3\check{\alpha}-3\check{\beta}$ &  \\ \hline
	$\langle \alpha+\beta \rangle =-1$, $\langle \lambda, \beta \rangle >0$ & $w(\lambda)-3\check{\alpha}-3\check{\beta}$ & $w(\lambda)-2\check{\alpha}-\check{\beta}$ \\ \hline
	$\langle \lambda, \beta \rangle=1$, \, $\langle \lambda, \alpha \rangle = -2$ & $w(\lambda)-3\check{\alpha}-3\check{\beta}$ & $w(\lambda)-2\check{\alpha}-2\check{\beta}$ \\ \hline
	$\langle \lambda, \alpha+\beta \rangle \geq 0$, $\langle \lambda, 2\alpha+\beta \rangle <-2$ & $w(\lambda)-2\check{\alpha}-2\check{\beta}$ &  \\ \hline
	$\langle \lambda, \alpha \rangle \leq -2$, $\langle \lambda, 2\alpha+\beta \rangle =$-2 or -1 & $w(\lambda)-2\check{\alpha}-\check{\beta}$ & $w(\lambda)-\check{\alpha}$ \\ \hline
	$\langle \lambda, \beta \rangle =1$, $\langle \lambda, \alpha \rangle = -1$ & $w(\lambda)-\check{\alpha}-\check{\beta}$ & $w(\lambda)-\check{\alpha}-\check{\beta}$ \\ \hline
	$\langle\lambda, 2\alpha+\beta \rangle \geq 0$,$\langle \lambda, \alpha \rangle \leq -1$ & $w(\lambda)-\check{\alpha}$ &  \\ \hline
	$\langle \lambda, \alpha \rangle \geq 0$ \, $\langle \lambda,\beta \rangle \geq 0$ & $w(\lambda)$ &  \\ \hline
	$\langle \lambda, \alpha \rangle \geq 0$, $\langle \lambda, 2\alpha+\beta \rangle <-2$ & $w(\lambda)-2\check{\alpha}-4\check{\beta}$ &  \\ \hline
	$\langle \lambda, \alpha \rangle >0$, $\langle \lambda, 2\alpha+\beta \rangle =$-2 or -1 & $w(\lambda)-\check{\alpha}-3\check{\beta}$ & $w(\lambda)-2\check{\alpha}-4\check{\beta}$ \\ \hline
	$\langle \lambda, 2\alpha+\beta \rangle \geq 0$, $\langle \lambda, \alpha+\beta \rangle <-1$ & $w(\lambda)-\check{\alpha}-3\check{\beta}$ &  \\ \hline
	$\langle \lambda, \alpha+ \beta \rangle = -1$, $\langle \lambda, 2\alpha+\beta \rangle \geq 2$ & $w(\lambda)-\check{\beta}$ & $w(\lambda)-\check{\alpha}-3\check{\beta}$ \\ \hline
	$\langle \lambda, \alpha+\beta \rangle = -1$, $\langle \lambda, 2\alpha+\beta \rangle =0$ & $w(\lambda)-\check{\beta}$ & $w(\lambda)-\check{\alpha}-2\check{\beta}$ \\ \hline
	$\langle \lambda, \alpha+\beta\rangle \geq 0$, $\langle \lambda, \beta \rangle <0$ & $w(\lambda)-\check{\beta}$ &  \\ \hline
	\end{tabular}
\end{center}
As an illustration, we do an example calculation in type $A_2$. Assume $\langle \lambda, \alpha \rangle \geq 0$ and $\langle \lambda, \alpha+\beta \rangle <-1$. Then we compute $R_{\alpha} R_{\beta} R_{\alpha}(\lambda)$. As a first step, we compute $\langle \lambda, \alpha \rangle$. By assumption $\langle \lambda, \alpha \rangle \geq 0$ so $R_{\alpha}(\lambda)=s_{\alpha}(\lambda)$. Thus \[R_{\alpha}R_{\beta}R_{\alpha}(\lambda)=R_{\alpha}R_{\beta}(s_{\alpha}(\lambda)).\] Next we compute $\langle s_{\alpha}(\lambda), \beta \rangle = \langle \lambda, s_{\alpha}(\beta) \rangle = \langle \lambda, \alpha + \beta \rangle$. By assumption, this is less than or equal to $-1$. Thus \[R_{\alpha}R_{\beta}(s_{\alpha}(\lambda))=R_{\alpha}(s_{\beta}(s_{\alpha}(\lambda))-\check{\beta}).\] Lastly we compute \[\langle s_{\beta}(s_{\alpha}(\lambda))-\check{\beta}, \alpha \rangle=\langle s_{\alpha}(\lambda)+\check{\beta}, s_{\beta}(\alpha) \rangle = \langle \lambda+ s_{\alpha}(\check{\beta}), s_{\beta}(\alpha+\beta) \rangle=\langle \lambda+\check{\alpha}+\check{\beta}, \beta \rangle.\] The two assumptions at the beginning on $\lambda$ force $\langle \lambda, \beta \rangle <-1$ so that $\langle \lambda+\check{\alpha}+\check{\beta}, \beta \rangle \leq -2+1<0$, so finally we see that $R_{\alpha} R_{\beta} R_{\alpha}(\lambda)=s_{\alpha}(s_{\beta}(s_{\alpha}(\lambda))-\check{\beta})-\check{\alpha}$. Simplifying this expression yields $R_{\alpha} R_{\beta} R_{\alpha}(\lambda)=s_{\alpha+\beta}(\lambda)-2\check{\alpha}-\check{\beta}$ as desired. Computing $R_{\beta}R_{\alpha}R_{\beta}$ proceeds in a similar fashion.

We do another example in type $B_2$ where the braid relation holds. Let $\langle \lambda, 2\alpha+\beta \rangle \geq 0$ and $\langle \lambda, \alpha+\beta \rangle <-1$. 
We first compute $R_{\alpha}R_{\beta}R_{\alpha}R_{\beta}(\lambda)$. We first calculate $\langle \lambda, \beta \rangle$ which by assumption is $<0$. 
Thus the first simplification is
 \[R_{\alpha}R_{\beta}R_{\alpha}R_{\beta}(\lambda)=R_{\alpha}R_{\beta}R_{\alpha}(s_{\beta}(\lambda)-\check{\beta}).\] 
Next we compute  $\langle s_{\beta}(\lambda)-\check{\beta}, \alpha \rangle = \langle \lambda+\check{\beta}, \beta+\alpha \rangle <0$.
So the next simplification is
 \[R_{\alpha}R_{\beta}R_{\alpha}R_{\beta}(\lambda)=R_{\alpha}R_{\beta}((s_{\alpha}(s_{\beta}\lambda)-\check{\beta})-\check{\alpha}).\] 
The next pairing $\langle s_{\alpha}(s_{\beta}\lambda)-\check{\beta})-\check{\alpha}, \beta \rangle >0$, 
so the next step is \[R_{\alpha}R_{\beta}R_{\alpha}R_{\beta}(\lambda)=R_{\alpha}s_{\beta}(s_{\alpha}(s_{\beta}\lambda)-\check{\beta})-\check{\alpha}).\] One can then calculate the last pairing, see that it is $\geq 0$, and obtain finally that \[R_{\alpha}R_{\beta}R_{\alpha}R_{\beta}(\lambda)=s_{\alpha}s_{\beta}(s_{\alpha}((s_{\beta}\lambda)-\check{\beta})-\check{\alpha})=s_{\alpha}s_{\beta}s_{\alpha}s_{\beta}(\lambda)-\check{\alpha}-3\check{\beta}.\]

Computing the other side, $R_{\beta}R_{\alpha}R_{\beta}R_{\alpha}(\lambda)$ is much the same; we mention that the first two pairings will be positive and the last two negative, leading to \[R_{\beta}R_{\alpha}R_{\beta}R_{\alpha}(\lambda)=s_{\beta}(s_{\alpha}(s_{\beta}s_{\alpha}(\lambda))-\check{\alpha})-\check{\beta}=s_{\beta}s_{\alpha}s_{\beta}s_{\alpha}(\lambda)-\check{\alpha}-3\check{\beta}.\] Thus, in this case the braid relation holds.
\end{proof}

\begin{remark}
A summary of the data from the tables in the above proof is as follows: the ``braid relations" for the $R$-operators hold in types $A_2$, $B_2$ everywhere except at a certain set of critical lines. In the following $w_0$ denotes the longest element in Weyl group of each type.
In type $A_2$, the braid relations fail precisely when $\langle \lambda, \alpha+\beta \rangle = -1$ and $\lambda \notin \frak{C}_{w_0}$. Similarly in $B_2$, the braid relations hold except when $\langle \lambda, \alpha+\beta \rangle =-1$ and $\lambda \notin \frak{C}_{w_0}$, or when $\langle \lambda, \beta+2\alpha \rangle = -1$ or $-2$ and $\lambda \notin \frak{C}_{w_0}$. Though we have not done all the calculations for $G_2$, it seems to follow the same pattern: in particular the braid relations appear to hold in all cases, except the following: $\langle \lambda, \alpha+2\beta \rangle = -1$, $\lambda \notin \frak{C}_{w_0}$, or $\langle \lambda, 3\beta+2\alpha \rangle = -1$ and $\lambda \notin \frak{C}_{w_0}$.
\end{remark}

\begin{example}
We do an example calculation in type $B_2$ where the braid relation fails. Let $\langle \lambda, \alpha+\beta \rangle = -1$ and let $\langle \lambda, 2\alpha+\beta \rangle \geq 2$. 
First we compute $R_{\alpha}R_{\beta}R_{\alpha}R_{\beta}$. As a first step we compute $\langle \lambda, \beta \rangle < 0$, 
so \[R_{\alpha}R_{\beta}R_{\alpha}R_{\beta}(\lambda)=R_{\alpha}R_{\beta}R_{\alpha}(s_{\beta}(\lambda)-\check{\beta}).\] 
Next we compute $\langle s_{\beta}(\lambda)-\check{\beta}, \alpha \rangle=\langle \lambda+\check{\beta}, \alpha+\beta \rangle$, 
which by assumption is $0$, so we have
 \[R_{\alpha}R_{\beta}R_{\alpha}R_{\beta}(\lambda)=R_{\alpha}R_{\beta}s_{\alpha}((s_{\beta}(\lambda)-\check{\beta}))=R_{\alpha}R_{\beta}(s_{\alpha}s_{\beta}(\lambda)-\check{\alpha}-\check{\beta}).\]
Computing \[\langle s_{\alpha}s_{\beta}(\lambda)-\check{\alpha}-\check{\beta}, \beta \rangle = \langle s_{\beta}(\lambda)-\check{\beta}, \beta+2\alpha \rangle =\langle \lambda +\check{\beta}, \beta+2\alpha \rangle \] 
which is forced to be $\geq 0$ by our assumptions, we see that the next step is $R_{\alpha}(s_{\beta}s_{\alpha}s_{\beta}(\lambda)-\check{\alpha}-\check{\beta})$. The last computation is $\langle s_{\beta}s_{\alpha}s_{\beta}(\lambda)-\check{\alpha}-\check{\beta}, \alpha \rangle$; our assumptions force this to be positive, so our final result is \[R_{\alpha}R_{\beta}R_{\alpha}R_{\beta}(\lambda)=s_{\alpha}s_{\beta}s_{\alpha}s_{\beta}(\lambda)-\check{\beta}\] as desired.
On the other hand, we compute $R_{\beta}R_{\alpha}R_{\beta}R_{\alpha}(\lambda)$. Since $\langle \lambda, \alpha \rangle >0$ we must have $R_{\beta}R_{\alpha}R_{\beta}R_{\alpha}(\lambda)=R_{\beta}R_{\alpha}R_{\beta}s_{\alpha}(\lambda)$. Next we see that $\langle s_{\alpha}(\lambda), \beta \rangle >0$, so we simplify further
 \[R_{\beta}R_{\alpha}R_{\beta}R_{\alpha}(\lambda)=R_{\beta}R_{\alpha}(s_{\beta}s_{\alpha}(\lambda)).\] 
 Our assumptions force $\langle s_{\beta}s_{\alpha}(\lambda), \alpha \rangle <0$,
and so the next simplification is \[R_{\beta}R_{\alpha}R_{\beta}R_{\alpha}(\lambda)=R_{\beta}(s_{\alpha}s_{\beta}s_{\alpha}(\lambda)-\check{\alpha}).\] Lastly we compute $\langle s_{\alpha}s_{\beta}s_{\alpha}(\lambda)-\check{\alpha}, \beta \rangle <0 $. So we finally see \[R_{\beta}R_{\alpha}R_{\beta}R_{\alpha}(\lambda)=s_{\beta}(s_{\alpha}s_{\beta}s_{\alpha}(\lambda)-\check{\alpha})-\check{\beta}=s_{\beta}s_{\alpha}s_{\beta}s_{\alpha}(\lambda)-\check{\alpha}-3\check{\beta}.\]
\end{example}

\subsection{Discussions on the moment polytope of $\overline{\X}_\lambda$}
\label{sect4.4}
For any $\lambda\in \check{P}$, we define the moment polytope ${\tt MP}(\lambda)$ of $\overline{\X}_\lambda$ as the convex hull of $\Psi(\lambda)$ in $\check{P}\otimes_\mathbb{Z}\bb{R}$. 
A part of our original motivation for this work was to explicitly understand the moment polytope  ${\tt MP}(\lambda)$, and understand whether all integral points (relative to $\lambda$) in ${\tt MP}(\lambda)$ appear in $\Psi(\lambda)$.

 Due to Theorem 4.1, if $\lambda \in \frak{C}_w$ for some $w\in W$, then the part of the moment polytope inside the chamber $\frak{C}_w$ is easy to describe. 
In general, to describe a polytope, it suffices to describe the vertices of this polytope. For the polytope  ${\tt MP}(\lambda)$, all vertices are some special points in $\Psi(\lambda)$. Fix a chamber $\frak{C}_y$. 
Let ${\tt M}_y(\lambda)$ denote the set of all coweights in $\Psi(\lambda)$ that are maximal in the chamber $\frak{C}_y$. 
Using Theorem 4.1 again, we see that in any given chamber $\frak{C}_y$ the coweights of $\Psi(\lambda)$ appearing in chamber $\frak{C}_{y}$ are precisely those $\mu$ such that $\lambda'-\mu$ is a sum of coroots in $y(\Phi^+)$, for some  $\lambda'\in {\tt M}_y(\lambda)$. 
Therefore the points in ${\tt M}_y(\lambda)$ are exactly the vertices of ${\tt MP}(\lambda)$ in the chamber $\frak{C}_{y}$.

We may use Proposition \ref{stability_prop} to 
obtain candidates of points in ${\tt M}_y(\lambda)$. All of these points are obtained via $R_{\beta_1} \dots R_{\beta_k}(\lambda)$ so that $\frak{C}_y=s_{\beta_1}\cdots s_{\beta_k}(\frak{C}_w )$, in other words $y=s_{\beta_1}\cdots s_{\beta_k}w$.
 However in general it seems very difficult to determine which path of $R$-operators will provide the maximal vectors.

 In the picture of the $A_2$ example below, we see five "paths" given by 
 \[ R_{\alpha}R_{\beta}R_{\alpha}( \lambda),\, R_{\beta}R_{\alpha}R_{\beta}(\lambda),\, R_{\alpha}R_{\alpha+\beta}R_{\beta}(\lambda),\, R_{\beta}R_{\alpha+\beta}R_{\alpha}(\lambda), \, R_{\alpha+\beta}(\lambda),\]
 where $\lambda=-3(\check{\alpha}+\check{\beta})$. Note that the last path $R_{\alpha+\beta}(\lambda)$ yields the maximal coweight $2(\check{\alpha}+\check{\beta})$, and all other candidates $\check{\alpha}+\check{\beta}$, $2 \check{\beta}+\check{\alpha}$ and $2 \check{\alpha}+\check{\beta}$ are in $\Psi(2 \check{\alpha}+2 \check{\beta})$. Thus far this pattern has held in all our rank 2 experiments; in any given chamber $\frak{C}_w$ we have only ever observed one maximal coroot $\mu$ such that all other candidates $\mu'=R_{\beta_1} \dots R_{\beta_k}(\lambda) \in \Psi(\mu)$. In $A_2$ it seems to be the case (after much tedious calculation) that $R_{\alpha}R_{\beta}R_{\alpha}(\lambda) \in \Psi(R_{s_{\alpha}(\beta)})$, so in $A_2$, to find the maximal candidates, one should replace $R_{\alpha}R_{\beta}R_{\alpha}$ with $R_{\alpha+\beta}$. On the other hand in $B_2$ there is an example where $R_{\alpha}(\lambda) \in \Psi(R_{\alpha+\beta}R_{\beta}R_{2\alpha+\beta}(\lambda))$, (note here $s_{\alpha}=s_{\alpha+\beta}s_{\beta}s_{2\alpha+\beta})$ so shorter expressions do not always do better.
 
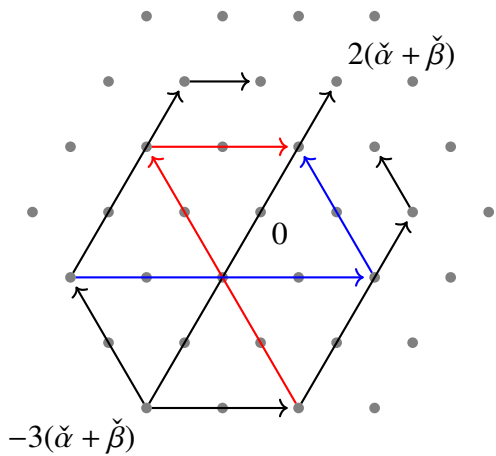
\begin{center}
	\begin{tikzpicture}[scale=1.00]
	\coordinate (0;0) at (0,0); 
\foreach \c in {1,...,3}{%
\foreach \i in {0,...,5}{%
\pgfmathtruncatemacro\j{\c*\i}
\coordinate (\c;\j) at (60*\i:\c);  
} }
\foreach \i in {0,2,...,10}{%
\pgfmathtruncatemacro\j{mod(\i+2,12)} 
\pgfmathtruncatemacro\k{\i+1}
\coordinate (2;\k) at ($(2;\i)!.5!(2;\j)$) ;}

\foreach \i in {0,3,...,15}{%
\pgfmathtruncatemacro\j{mod(\i+3,18)} 
\pgfmathtruncatemacro\k{\i+1} 
\pgfmathtruncatemacro\l{\i+2}
\coordinate (3;\k) at ($(3;\i)!1/3!(3;\j)$)  ;
\coordinate (3;\l) at ($(3;\i)!2/3!(3;\j)$)  ;
 }

 \foreach \i in {0,...,6}{%
\pgfmathtruncatemacro\k{\i}
\pgfmathtruncatemacro\l{15-\i}
 \pgfmathtruncatemacro\k{9-\i} 
 \pgfmathtruncatemacro\l{mod(12+\i,18)}   
 \pgfmathtruncatemacro\k{12-\i} 
 \pgfmathtruncatemacro\l{mod(15+\i,18)}   
}    
\fill [gray] (0;0) circle (2pt);
 \foreach \c in {1,...,3}{%
 \pgfmathtruncatemacro\k{\c*6-1}    
 \foreach \i in {0,...,\k}{%
   \fill [gray] (\c;\i) circle (2pt);}}  
\draw[->,black,thick,shorten >=4pt,shorten <=2pt](3;12)--(3;14);
\draw[->,black,thick,shorten >=4pt,shorten <=2pt](3;12)--(3;10);
\draw[->,black,thick,shorten >=4pt,shorten <=2pt](3;14)--(2,0);
\draw[->,black,thick,shorten >=4pt,shorten <=2pt](3;10)--(2;4);
\draw[->,black,thick,shorten >=4pt,shorten <=2pt](2;0)--(2;1);
\draw[->,black,thick,shorten >=4pt,shorten <=2pt](2;4)--(2;3);
\draw[->,black,thick,shorten >=4pt,shorten <=2pt](3;12)--(2;2);
\draw[->,red,thick,shorten >=4pt,shorten <=2pt](3;14)--(2;5);
\draw[->,red,thick,shorten >=4pt,shorten <=2pt](2;5)--(1;1);
\draw[->,blue,thick,shorten >=4pt,shorten <=2pt](3;10)--(2;11);
\draw[->,blue,thick,shorten >=4pt,shorten <=2pt](2;11)--(1;1);
\draw (3;12) node[anchor=north east] {$-3(\check{\alpha}+\check{\beta})$};
\draw (0;0) node[anchor=north west] {$0$};
\draw (2;2) node[anchor=south west] {$2(\check{\alpha}+\check{\beta})$};
	\end{tikzpicture}
\end{center}

One further remark is that in the above example, every integral coweight inside ${\tt MP}(\lambda)$ is indeed contained in $\Psi(\lambda)$, where $\lambda=-3(\check{\alpha}+\check{\beta})$. If it were true in general that for any $\lambda$ and for any $y \in W$, ${\tt M}_y(\lambda)$ contains at most one element (in other words, if each chamber has a unique ``maximal" candidate if it exists), then this would imply the following property: for every integral point $\mu$ inside ${\tt MP}(\lambda)$, $\mu \in \Psi(\lambda)$. This property holds for $A_1 \times A_1$ and $A_2$ by naive calculation, and has held in every other example we have tried for $B_2$. We would be interested in a proof or disproof of this property for general root systems. When $\lambda$ is dominant, $\overline{\X}_\lambda=\overline{\Gr}_\lambda$. In this case, it is well-known that ${\tt MP}(\lambda)$ is the convex hull of $\{w(\lambda) \li w\in W\}$, and for every weight $\mu\in {\tt MP}(\lambda)$  such that $\lambda-\mu\in \check{Q}$, $\mu\in \Psi(\lambda)$.
The study of the moment polytopes of certain subvarieties in affine Grassmannian very often leads to interesting applications in representation theory, see \cite{An,Kam}.

\end{document}